\documentclass[a4paper, 12pt]{amsart}
\usepackage{amsmath,amssymb,amsthm}
\usepackage{amscd}
\usepackage{array,enumerate}
\newtheorem{Thm}{Theorem}[section]
\newtheorem{Prop}[Thm]{Proposition}
\newtheorem{Lem}[Thm]{Lemma}
\newtheorem{Cor}[Thm]{Corollary}
\newtheorem{Thmint}{Theorem}[section]

\newtheorem{MThm}{Main}

\newtheorem{Corint}[Thmint]{Corollary}

\theoremstyle{definition}
\newtheorem{Rem}[Thm]{Remark}

\newtheorem{Def}[Thm]{Definition}
\newtheorem{Exm}[Thm]{Example}

\newcommand{\Cs}{C$^\ast$}
\newcommand{\Ws}{W$^\ast$}
\newcommand{\id}{\mbox{\rm id}}
\newcommand{\Homeo}{\mathop{{\rm Homeo}}}
\newcommand{\rg}{\mathop{{\mathrm C}_{\mathrm r}^\ast}}

\newcommand{\rc}{\mathop{\rtimes _{\mathrm r}}}
\newcommand{\votimes}{\mathop{\bar{\otimes}}}
\newcommand{\motimes}{\mathop{{\otimes}_{\rm max}}}
\newcommand{\bigvotimes}{\mathop{\bar{\bigotimes}}}
\newcommand{\vnc}{\mathop{\bar{\rtimes}}}
\newcommand{\ac}{\mathop{\rtimes _{\mathrm{alg}}}}

\DeclareMathOperator{\supp}{supp}
\DeclareMathOperator{\PSL}{PSL}
\DeclareMathOperator{\SL}{SL}
\DeclareMathOperator{\Cso}{{\rm C}^\ast}
\DeclareMathOperator{\Wso}{{\rm W}^\ast}
\DeclareMathOperator{\ncl}{norm-closure}
\DeclareMathOperator{\stcl}{\sigma-strong-closure}
\DeclareMathOperator{\sslim}{\ast-strong-}

\DeclareMathOperator{\im}{im}
\DeclareMathOperator{\bigfp}{\lower0.25ex\hbox{\LARGE $\ast$}}
\newcommand{\ad}{\mathop{\rm ad}}
\title[Intermediate operator algebras vs. intermediate extensions]{\small Complete descriptions of intermediate operator algebras by intermediate extensions
of dynamical systems}

\author{Yuhei Suzuki}
\address{Graduate school of mathematics, Nagoya University, Chikusaku, Nagoya, 464-8602, Japan}
\email{yuhei.suzuki@math.nagoya-u.ac.jp}
\begin{document}
\begin{abstract}
Practically and intrinsically, inclusions of operator algebras are of fundamental interest.
The subject of this paper is intermediate operator algebras of inclusions. 
There are two previously known theorems which naturally and completely describe all intermediate operator algebras: the Galois Correspondence Theorem and the Tensor Splitting Theorem. Here we establish the third, new complete description theorem which gives a canonical bijective correspondence between intermediate operator algebras and intermediate extensions of dynamical systems.
One can also regard this theorem as a crossed product splitting theorem, analogous to the Tensor Splitting Theorem. We then give concrete applications, particularly to maximal amenability problem and a new realization result of intermediate operator algebra lattice.
\end{abstract}
\maketitle
\section{Introduction}\label{Sec:intro}

In operator algebra theory, it is a fundamental problem to consider inclusions of operator algebras.
They naturally arise from many subjects,
including knot theory \cite{Jon91}, operator theory \cite{Arv}, and algebraic quantum field theory \cite{Kaw}.
Certainly inclusions also have intrinsic interest; since Jones'
index theorem \cite{Jon}, subfactor theory becomes one of the central subject in operator algebra theory. the Connes embedding problem \cite{Con}
(see \cite{OzaQ} for a survey)
is one of the most important open problem in operator algebra theory---
the highlight is Kirchberg's theorem \cite{Kir93}
which reveals unexpected connections of this problem with other important open problems.
Also, not just being of intrinsic interest, but also nice inclusions often play crucial roles
in the classification/structure theory of operator algebras.
Here we list only a few of celebrated results in which inclusions/embeddings play crucial roles:
Connes' classification of injective factors \cite{Con} (cf.~\cite{OzaQ}),
Popa's deformation/rigidity theory \cite{Pop06}, \cite{PopICM},
the Akemann--Ostrand property for hyperbolic groups \cite{HG} (see \cite{Oza04}, \cite{OP} for significant applications),
characterizations of \Cs-simplicity of discrete groups \cite{KK}, \cite{BKKO} (cf.~ \cite{Ham85}, \cite{Ham11}, \cite{Oza07}).
These facts support the importance of the study of inclusions of operator algebras,
particularly to give detailed analysis of inclusions and
to construct inclusions with interesting properties.

The subject of the paper is the complete description problem of
inclusions of operator algebras.
Here we briefly recall some known major results on this subject.
Galois correspondence results state that for compact groups $G$ and for discrete groups $\Gamma$, under certain assumptions,
there is a natural bijective correspondence between intermediate von Neumann algebras
of $M^G \subset M$, $M \subset M \vnc \Gamma$
and closed subgroups of $G$, $\Gamma$ respectively.
This reduces the complete description problem of intermediate operator algebras---usually very hard and almost impossible---to that of closed subgroups---algebraic, tractable problem.
We refer the reader to \cite{ILP} for backgrounds and history of this subject.
The fundamental work of Izumi--Longo--Popa \cite{ILP}
finally established the Galois correspondence for all minimal actions of $G$
and all outer actions of $\Gamma$ on arbitrary factors (\cite{ILP}, Theorems 3.15 and 3.13). The \Cs-algebra analogues of Galois correspondences are also studied by many authors.
Izumi has shown the Galois correspondence
for outer actions of finite groups on $\sigma$-unital simple \Cs-algebras, both as a compact group and as a discrete group
(\cite{Izu}, Corollary 6.6).

Another description result of intermediate operator algebras
arises from tensor products.
Ge--Kadison's tensor splitting theorem \cite{GK} states that
for any factor $M$ and for any von Neumann algebra $N$,
any intermediate von Neumann algebra of the inclusion $M \subset M \votimes N$
splits into $M\votimes N_0$ for some von Neumann subalgebra $N_0 \subset N$.
An analogous result for simple \Cs-algebras is independently proved
by Zacharias \cite{Zac} and Zsido \cite{Zsi}.
This reduces the problem on intermediate operator algebras
to that on operator subalgebras.
We later combine these results with our theorems to obtain interesting inclusions. See Examples \ref{Exm:minnuc} to \ref{Exm:Lat}.

In this paper, we establish a new complete description result from a different direction.
Instead of changing groups, we change the coefficients in the crossed products.
More precisely, we establish the bijective correspondence between intermediate operator algebras
and intermediate extensions of dynamical systems for certain crossed product inclusions.
This reduces the complete description problem of intermediate operator algebras
to that of intermediate extensions, which is geometric/ergodic theoretical, relatively tractable, with various interesting examples.
This establishes a new interaction between operator algebras and dynamical systems.
\begin{MThm}\label{Thm:Main}
We obtain the following bijective correspondences.
\begin{description}
\item[\Cs-case (Theorem \ref{Thm:MainC})]
Let $\Gamma$ be a discrete group satisfying the AP \cite{HK}.
Let $\alpha \colon \Gamma \curvearrowright X$ be a free action on a locally compact space $X$.
Let $\beta\colon \Gamma \curvearrowright Y$ be an extension of $\alpha$.
Let $\pi \colon Y \rightarrow X$ be a proper factor map.
Then the map 
\[C_0(Z) \mapsto C_0(Z)\rc \Gamma \]
gives a bijective correspondence between the set of intermediate extensions
$C_0(X) \subset C_0(Z) \subset C_0(Y)$ of $\pi$
and that of intermediate \Cs-algebras
of 
$C_0(X)\rc \Gamma \subset C_0(Y)\rc \Gamma$.
\item[\Ws-case (Theorem \ref{Thm:MainW})]
Let $\Gamma$ be a countable group.
Let $\alpha \colon \Gamma \curvearrowright (X, \mu)$ be an essentially free non-singular action on a standard probability space.
Let $\beta \colon \Gamma \curvearrowright (Y, \nu)$ be an extension of $\alpha$.
Let $\pi \colon Y \rightarrow X$ be a factor map.
Then the map 
\[L^\infty(Z) \mapsto L^\infty(Z)\vnc \Gamma \]
gives a bijective correspondence between the set of intermediate extensions $L^\infty(X) \subset L^\infty(Z)\subset L^\infty(Y)$ of $\pi$
and that of intermediate von Neumann algebras
of 
$L^\infty(X)\vnc \Gamma \subset L^\infty(Y)\vnc \Gamma$.
\end{description}
\end{MThm}
To the author's knowledge, this is the first result in this direction.
(However we point out that Hamachi--Kosaki \cite{HaK}
studied some of these inclusions in the context of subfactor theory.)
The statement is almost complete;
we cannot remove the freeness condition (even cannot be weakened to topological freeness (Proposition \ref{Prop:topfree})), and the \Cs-algebra case fails for non-exact groups (Proposition \ref{Prop:nonexact}).
Typical examples of free actions arise as the left translation actions associated to group embeddings into locally compact groups.
Also, we have shown in \cite{Suzst} that (countable, non-torsion, exact) groups admit many minimal free actions on various compact spaces (\cite{Suzst}, Theorem B.1).
Essentially free actions naturally appear as translation
actions on homogeneous spaces and as Bernoulli shift actions; see \cite{Zim}.

As a special application, we give new examples
of maximal amenable subfactors.
We shortly recall a background of the maximal amenability problem.
The first observation on maximal amenability
is due to Fuglede--Kadison \cite{FK}, where they observed that
all II$_1$-factors contain a maximal AFD II$_1$ subfactor.
In the seminal paper \cite{Pop}, Popa invented a powerful strategy
(now called Popa's asymptotic othogonality property)
which in particular confirms the maximal amenability of
the inclusion $L \mathbb{Z} \subset L\mathbb{F}_2$; $\mathbb{Z}=\langle a \rangle \subset\langle a, b \rangle =\mathbb{F}_2$.
Remarkably, this settles Kadison's problem
asking if every amenable von Neumann subalgebra of a type II$_1$ factor
admits an intermediate AFD II$_1$ factor.
After his discovery, following basically the same line, many researchers study and construct
various examples of maximal amenable von Neumann subalgebras.
We refer the reader to \cite{Hou} and references therein for some of them.
Recently, based on the study of central states and \Cs-algebra theory, Boutonnet--Carderi \cite{BC} gave a new method
to confirm maximal amenability of certain group von Neumann algebra inclusions
$L\Lambda \subset L\Gamma$.
Our Main Theorem gives yet another method to
provide maximal amenable subalgebras.
Combining the Main Theorem and Margulis' factor theorem \cite{Mar78},
we obtain new examples of maximal amenable subfactors of full factors of ergodic theoretical origin.
Notably, the lattices of intermediate von Neumann algebras of these inclusions
are completely computable (cf.~ Proposition 7.76 of \cite{Kna}, Example \ref{Exm:Mar}).
For simplicity, here we restrict the statement to simple Lie groups.
See Corollary \ref{Cor:Mar} for the full statement.
\begin{Corint}\label{Corint:Mar}
Let $G$ be a connected simple Lie group with trivial center of real rank at least $2$.
Let $\Gamma$ be a lattice in $G$.
Let $P$ be a minimal parabolic subgroup of $G$.
Let $Q$ be a proper intermediate closed group of $P \subset G$.
Then the map
\[R\mapsto L^\infty(G/R)' \vnc\Gamma {\rm ~(where~the~ commutant~ is~ taken~ in~} \mathcal{B}(L^2(G/P)))\]
gives a lattice isomorphism from
the lattice of intermediate closed groups $R$ of $P\subset Q$
onto that of intermediate von Neumann algebras of $L^\infty(G/P)\vnc \Gamma \subset L^\infty(G/Q)'\vnc \Gamma$.
In particular $L^\infty(G/P)\vnc \Gamma$ is a maximal amenable
subalgebra of the non-amenable factor $L^\infty(G/Q)' \vnc \Gamma $.
\end{Corint}

We also discuss a non-commutative generalization of the Main Theorem (Theorems \ref{Thm:NCMainC} and \ref{Thm:NCMainW}).
This produces further interesting examples of inclusions (Examples \ref{Exm:NCBC} to \ref{Exm:Lat}).
As a byproduct, we obtain the following realization theorem.
\begin{Corint}[Example \ref{Exm:Lat}]\label{Corint:Lat}
Let $G$ be a locally compact second countable group.
Then one can realize the closed subgroup lattice of $G$ as the lattice of intermediate von Neumann algebras of an irreducible subfactor $N \subset M$.
\end{Corint}
As another byproduct, we construct
``exotic'' endomorphisms on all Kirchberg algebras.
Here recall that a simple separable nuclear purely infinite \Cs-algebra
is said to be a Kirchberg algebra.
One of the highlight of the classification theory of \Cs-algebras
is the Kirchberg--Phillips classification theorem \cite{Kir94}, \cite{Phi}.
This deep theorem states that Kirchberg algebras are completely classified by KK-theory.
This motivates further study on Kirchberg algebras,
especially on their symmetrical structure and inclusions (see e.g., \cite{Nak}, \cite{Izu}, \cite{Izu04}, \cite{IM}).
We point out that this subject also has a strong connection to the index theory (see e.g., \cite{Izu95}, \cite{Izu98}).

We say that an endomorphism $\sigma$ of a \Cs-algebra $A$
is without expectation if the inclusion $\sigma(A) \subset A$
does not admit a conditional expectation.
(This condition in particular means infiniteness of the index \cite{Wat}.)
Similarly, we say that $\sigma$
is without proper intermediate \Cs-algebras if the inclusion $\sigma(A) \subset A$
has no proper intermediate \Cs-algebras.
\begin{Corint}[Theorem \ref{Thm:End}]\label{Corint:End}
Every Kirchberg algebra admits an endomorphism without proper intermediate \Cs-algebras or expectation.
\end{Corint}
To the author's knowledge, Corollary \ref{Corint:End} is the first result in this direction.
\subsection*{Organization of the paper}
In Section \ref{Sec:C}, we prove the Main Theorem for \Cs-algebras.
We also demonstrate how freeness and the AP
are essential by constructing counterexamples.
Section \ref{Sec:W} is devoted for von Neumann algebras.
Corollary \ref{Corint:Mar} then follows from Margulis' factor theorem.
In Section \ref{Sec:NC}, we give a non-commutative generalization of the Main Theorem (Theorems \ref{Thm:NCMainC} and \ref{Thm:NCMainW}).
We then provide some applications, including Corollary \ref{Corint:Lat}.
In Section \ref{Sec:Pf}, we prove
Corollary \ref{Corint:End}.
To construct the desired endomorphism,
we first construct an appropriate actions of the infinite rank free group
on a Kirchberg algebra.

\subsection*{Lattices associated to inclusions}
For an inclusion $A \subset B$ of \Cs-algebras,
the set of intermediate \Cs-algebras of $A \subset B$
has the lattice operation
\[ C \vee D := \Cso(C, D),\]
\[ C \wedge D := C \cap D.\]
Similarly, for a closed subgroup of a topological group
and for an inclusion of von Neumann algebras,
there is a natural lattice structure on the set of intermediate objects.

\section{Intermediate \Cs-algebras and topological dynamical systems}\label{Sec:C}
In this section, we prove the Main Theorem for \Cs-algebras (Theorem \ref{Thm:MainC}).

For the basic facts and notations on \Cs-algebras, we refer the reader to \cite{BO}.
Here we only fix some notations.
Let $\Gamma$ be a discrete group.
For a $\Gamma$-\Cs-algebra $A$,
for $s \in \Gamma$,
denote by $u_s$ the implementing unitary element for $s$ in the multiplier algebra
$\mathcal{M}(A \rc \Gamma)$.
Denote by $E \colon A \rc \Gamma \rightarrow A$
the canonical conditional expectation
$E(x u_s) =x \delta_{e}(s)$; $x\in A, s\in \Gamma$.
For $a \in A \rc \Gamma$ and $s\in \Gamma$,
denote by $E_s(a)$ the $s$-th coefficient $E(a u_s^\ast)$ of $a$.
For a family $S_1, \ldots, S_n$ of subsets or elements in a \Cs-algebra $A$,
denote by $\Cso(S_1, \ldots, S_n)$ the \Cs-subalgebra of $A$ generated by
$S_1, \ldots, S_n$.
The symbol `$\otimes$' stands for the minimal tensor product of \Cs-algebras.

We next recall some terminologies of topological dynamical systems.
For a locally compact (Hausdorff) space $X$, by the Gelfand duality,
there is a bijective correspondence between group actions
on $X$ and those on $C_0(X)$.
We use the same symbol for the corresponding action.
For instance, for $\alpha \colon \Gamma \curvearrowright X$,
$f\in C_0(X)$, $s\in \Gamma$,
we set $\alpha_s(f)=f \circ \alpha_{s}^{-1}$.
For two actions $\alpha \colon \Gamma \curvearrowright X$
and $\beta \colon \Gamma \curvearrowright Y$ of a discrete group $\Gamma$
on locally compact spaces $X, Y$,
$\beta$ is said to be an extension of $\alpha$
if there is an equivariant quotient map $\pi \colon Y \rightarrow X$.
We refer to $\pi$ as a factor map from $\beta$ onto $\alpha$.
We also refer to $\pi$ as an extension. Note that the choice of a factor map is
not necessary unique.
Hence giving an extension $\pi$
has more information than just stating ``$\beta$ is an extension of $\alpha$.''
A factor map is said to be proper if the preimage of any compact set is compact.
Let $\pi \colon Y \rightarrow X$ be an extension.
An intermediate extension of $\pi$
is a triplet $(\gamma, p, q)$
where $\gamma \colon \Gamma \curvearrowright Z$ is an action on a locally compact space,
$p\colon Y \rightarrow Z$, $q\colon Z \rightarrow X$ are factor maps
with $q\circ p =\pi$.
Note that when $\pi$ is proper, so are $p, q$.
We identify two intermediate extensions
$(\gamma, p, q)$ and $(\gamma' \colon \Gamma \curvearrowright Z', p', q')$ if
there is a $\Gamma$-homeomorphism
$h \colon Z \rightarrow Z'$ satisfying
$h\circ p=p', q= q' \circ h$.
For a proper extension $\pi$, after this identification,
the Gelfand duality gives a bijective correspondence between
the intermediate extensions of $\pi$
and the intermediate $\Gamma$-\Cs-algebras
of the associated inclusion $C_0(X)\subset C_0(Y)$.
This induces a lattice structure on
the set of equivalence classes of intermediate extensions of $\pi$.

We start by the following simple lemma.
\begin{Lem}\label{Lem:ccpC}
Let $A$ be a \Cs-algebra.
Let $\Phi \colon A\rightarrow A$ be a contractive completely positive map.
Let $S$ be a subset of $A$.
Define
\[\ell^2( S)_1:=\left\{(s_i)_{i=1}^\infty \in S^{\mathbb{N}}:{\rm the ~series~}S:=\sum_{i=1}^\infty s_i^\ast s_i {\rm ~converges~in~norm~ and~ } \|S\|\leq 1\right\}.\] 
For $a\in A$,
set \[\mathfrak{N}(a; S):=\ncl\left\{\sum_{i=1}^\infty s_i^\ast a s_i \colon (s_i)_{i=1}^\infty \in \ell^2( S)_1\right\}.\]
Then the set
\[\mathfrak{N}(\Phi; S) :=\{a\in A: \Phi(a)\in \mathfrak{N}(a; S)\}\]
is norm closed in $A$.
\end{Lem}
\begin{proof}
Observe that for any $(s_i)_{i=1}^\infty \in \ell^2( S)_1$,
the map $\Psi\colon A \rightarrow A$
given by $\Psi(a):=\sum_{i=1}^\infty s_i^\ast a s_i$
defines a contractive completely positive map.
With this observation, standard norm estimations complete the proof.
\end{proof}

The following lemma is well-known.
For the reader's convenience, we include the proof.
\begin{Lem}\label{Lem:freet}
Let $\alpha \colon \Gamma \curvearrowright X$ be a free action on a locally compact space.
Then for any finite subset $F$ of $\Gamma \setminus\{1\}$
and for any compact subset $K$ of $X$,
there is a finite open cover $U_1, \ldots, U_n$ of $K$
satisfying 
$\alpha_s(U_i) \cap U_i=\emptyset$ for all $i=1, \ldots, n$ and $s\in F$.
\end{Lem}
\begin{proof}
Since $\alpha$ is free, for any $x\in X$ and $s\in F$,
we have $\alpha_s(x)\neq x$.
Since $\alpha$ is continuous,
for any $x\in X$, one can choose an open neighborhood $U_x$
of $x$ with $\alpha_s(U_x) \cap U_x =\emptyset$ for all $s\in F$.
Obviously $\mathcal{U}:=(U_x)_{x\in K}$ forms an open cover of $K$.
Now take a finite subcover of $\mathcal{U}$ for $K$.
\end{proof}
Haagerup--Kraus \cite{HK} introduced the approximation property (AP)
for discrete (more generally, for locally compact) groups.
The AP is weaker than weak amenability, stronger than exactness,
and it is preserved under many operations.
Variously many groups are known to have the AP, but not all exact
groups have the AP \cite{LS} (e.g., $\SL(3, \mathbb{Z})$).
See \cite{HK} and Section 12 of \cite{BO} for details.
Haagerup--Kraus \cite{HK} related the AP with the validity of the Fubini-type theorem
for the reduced group operator algebras.
In \cite{Suz17}, we observed that the AP is also useful
to determine the position of an element in the reduced crossed product (Proposition 3.4 in \cite{Suz17}).
We need this result, which is the reason of the AP requirement in the following theorem.
\begin{Thm}[The Main Theorem, \Cs-case]\label{Thm:MainC}
Let $\Gamma$ be a discrete group satisfying the AP. Let $X, Y$ be locally compact spaces.
Let $\alpha \colon \Gamma \curvearrowright X$,
$\beta \colon \Gamma \curvearrowright Y$ be free actions.
Let $\pi \colon Y \rightarrow X$ be a proper factor map.
We regard $C_0(X)$ as a $\Gamma$-\Cs-subalgebra of $C_0(Y)$ via $\pi$.
Then the map 
\[C_0(Z) \mapsto C_0(Z)\rc \Gamma \]
gives a lattice isomorphism between the lattice of intermediate extensions $Z$ of $\pi$
and that of intermediate \Cs-algebras
of 
$C_0(X)\rc \Gamma \subset C_0(Y)\rc \Gamma$.
\end{Thm}
\begin{proof}
Let $a \in C_{\rm c}(Y) \rtimes_{\rm alg} \Gamma:= {\rm span}\{f u_s : f\in C_{\rm c}(Y), s\in \Gamma\}$.
We first show that $E(a)\in \mathfrak{N}(a; C_0(X))$.
Define \[F:= \{ s\in \Gamma \setminus\{e\}: E_s(a)\neq 0\}.\]
By the choice of $a$, $F$ is finite.
By Lemma \ref{Lem:freet}, one can choose a finite open cover $U_1, \ldots, U_n$ of $\pi(\supp(E(a))) \subset X$
satisfying $\alpha_s(U_i) \cap U_i=\emptyset$ for all $i$ and $s\in F$.
Choose non-negative functions $f_1, \ldots, f_n$ in $C_{\rm c}(X)$
satisfying $\supp(f_i) \subset U_i$ for each $i$, $\| \sum_{i=1}^n f_i^2\|\leq 1$, and
\[\sum_{i=1}^n f_i^2(y)=1 {\rm~for~all~}y \in \supp(E(a)).\]
By the choice of $U_i$, for any $s\in F$ and any $i$, we have
\[f_i u_s f_i= f_i \alpha_s(f_i) u_s=0.\]
This yields \[\sum_{i=1}^n f_i a f_i=\sum_{i=1}^n f_i^2E(a)= E(a)\]
as desired.

By Lemma \ref{Lem:ccpC}, we have
$\mathfrak{N}(E; C_0(X)) = C_0(Y) \rc \Gamma$.
Consequently
$E(a) \in \Cso(C_0(X), a)$ for all $a \in C_0(Y)\rc \Gamma$.
In particular, for any intermediate \Cs-algebra
$A$ of $C_0(X) \rc \Gamma \subset C_0(Y) \rc \Gamma$,
we have $E(A) \subset A$.
(We remark that we do not use the assumption on $\Gamma$ for this conclusion.)
Proposition 3.4 in \cite{Suz17} then yields $A=E(A) \rc \Gamma$.
\end{proof}
\begin{Rem}\label{Rem:nonsa}
It would be worth remarking
that the arguments in Theorem \ref{Thm:MainC}
does not use self-adjointness.
Thus the map
\[\mathcal{A} \mapsto \mathcal{A}\rc \Gamma:=\overline{\rm span}\{au_s:a\in \mathcal{A}, s\in \Gamma\}\]
gives a lattice isomorphism between the lattice of intermediate norm closed $\Gamma$-algebras of $C_0(X) \subset C_0(Y)$
and that of intermediate norm closed algebras of $C_0(X) \rc \Gamma \subset C_0(Y) \rc \Gamma$.
\end{Rem}
\begin{Rem}
The assumption on $\Gamma$ in Theorem \ref{Thm:MainC}
is not necessary when every intermediate extension $Z$ of $\pi$
admits a $\Gamma$-conditional expectation
$C_0(Y) \rightarrow C_0(Z)$.
Indeed, by Exercise 4.1.4 of \cite{BO},
any $\Gamma$-conditional expectation
$E\colon C_0(Y)\rightarrow C_0(Z)$ extends to
a conditional expectation
$\tilde{E}\colon C_0(Y)\rc \Gamma \rightarrow C_0(Z) \rc \Gamma$
satisfying $\tilde{E}(f u_s)=E(f)u_s$ for all $f\in C_0(Y)$ and $s\in \Gamma$.
This implies
\[C_0(Z)\rc \Gamma = \{ a\in C_0(Y) \rc \Gamma: E_s(a) \in C_0(Z) {\rm~for~all~}s\in \Gamma\}.\]
\end{Rem}

It is interesting to compare Theorem \ref{Thm:MainC} with the ideal structure
results of the reduced crossed product \Cs-algebras \cite{KT}, \cite{AS} (see also \cite{KK}, \cite{BKKO}).
Unlike these results, our argument requires
the freeness rather than just the topological freeness for actions.
(Recall that an action $\alpha \colon \Gamma \curvearrowright X$ is said to be topologically free if
for any $\Gamma$-invariant closed subset $Y$ of $X$
and for any $s\in \Gamma \setminus \{e\}$,
the set of fixed points of $\alpha_s$ in $Y$ has empty interior in $Y$.)
In fact, as the following proposition shows, the freeness assumption is essential
in Theorem \ref{Thm:MainC}.
\begin{Prop}\label{Prop:topfree}
Let $\Gamma$ be a discrete group.
Let $\alpha \colon \Gamma \curvearrowright X$ be a non-free action
on a locally compact space.
Then there is a free action $\beta \colon \Gamma \curvearrowright Y$
and a proper extension $\pi \colon Y \rightarrow X$ satisfying the following property.
There is an intermediate \Cs-algebra $A$ of $C_0(X) \rc \Gamma \subset C_0(Y) \rc \Gamma$ not of the form $C_0(Z) \rc \Gamma$
for any intermediate extension $Z$ of $\pi$.
\end{Prop}
\begin{proof}
Take $s\in \Gamma \setminus \{e\}$ and $x\in X$
satisfying $sx=x$. Let $\Lambda$ be the subgroup of $\Gamma$ generated by $s$.
Let $p \colon \Gamma / \Lambda \rightarrow X$
be the map given by $p(s\Lambda):=sx$.
Define
\[\iota \colon C_0(X) \rightarrow C_0(X) \oplus \ell^\infty(\Gamma /\Lambda)\]
to be $\iota(f):= (f, f \circ p)$.
Let $C_0(Y_0)$ denote the \Cs-subalgebra of $C_0(X) \oplus \ell^\infty (\Gamma/ \Lambda)$
generated by $\iota(C_0(X))$ and $0 \oplus c_0(\Gamma/ \Lambda)$.
We identify $C_0(X)$, $c_0(\Gamma /\Lambda)$ with these $\Gamma$-\Cs-subalgebras of $C_0(Y_0)$ (not with $C_0(X) \oplus 0$) respectively.
As $C_0(X) C_0(Y_0)$ is norm dense in $C_0(Y_0)$,
the inclusion $C_0(X) \subset C_0(Y_0)$
induces a proper extension $q \colon Y_0 \rightarrow X$.
Let $\Gamma \curvearrowright \beta \Gamma$ be the left translation action 
on the Stone--\v{C}ech compactification.
Let $\beta \colon \Gamma \curvearrowright Y:= \beta \Gamma \times Y_0$ denote the diagonal action.
Note that $\beta$ is free (see the proof of Lemma 2.3 in \cite{SuzMin} for instance).
Denote by $r \colon Y \rightarrow Y_0$ the projection onto the second coordinate.
Put $\pi := q \circ r \colon Y \rightarrow X$.

We show that $\pi$ is the desired extension.
To see this, observe that $c_0(\Gamma/\Lambda) \rc \Gamma$ is not simple.
(Indeed $\delta_\Lambda [c_0(\Gamma/\Lambda) \rc \Gamma] \delta_{\Lambda}
\cong \rg(\Lambda)$, where $\delta_\Lambda \in c_0(\Gamma/\Lambda)$ denotes the Dirac function
at $\Lambda \in \Gamma / \Lambda$.)
Choose a proper ideal $I$ of $c_0(\Gamma/\Lambda) \rc \Gamma$.
Note that $I \cap c_0(\Gamma/\Lambda)=0$, $E(I)=c_0(\Gamma/\Lambda)$
as $c_0(\Gamma/\Lambda)$ has no proper $\Gamma$-ideal.
Now define
\[A := C_0(X) \rc \Gamma + I \subset C_0(Y) \rc \Gamma.\]
As $C_0(X) \rc \Gamma$ is contained in the idealizer of $I$, $A$ is a \Cs-algebra.
Also,
\[A \cap C_0(Y) = C_0(X) \subsetneq C_0(X) + c_0(\Gamma/\Lambda) =E(A).\]
Therefore $A$ is not of the form $C_0(Z) \rc \Gamma$ for any intermediate extension $Z$ of $\pi$.
\end{proof}

We also remark that Theorem \ref{Thm:MainC} fails for non-exact groups
(\cite{Gro}, \cite{Osa14}).
\begin{Prop}\label{Prop:nonexact}
Let $\Gamma$ be a countable non-exact group.
Then there are compact spaces $X, Y$, free actions $\alpha \colon \Gamma \curvearrowright X$, $\beta \colon \Gamma \curvearrowright Y$, and a factor map $\pi \colon Y \rightarrow X$ satisfying the following property.
There is an intermediate \Cs-algebra of the inclusion
$C(X)\rc \Gamma \subset C(Y) \rc \Gamma$
not of the form $C(Z) \rc \Gamma$ for any
intermediate extension $Z$ of $\pi$.
\end{Prop}
\begin{proof}
Take a unital separable $\Gamma$-\Cs-subalgebra
$C(X)$ of $\ell^\infty(\Gamma)=C(\beta\Gamma)$ whose associated action
$\Gamma \curvearrowright X$ is free (see Lemma 2.3 of \cite{SuzMin}).
By replacing $C(X)$ by $C(X)+ c_0(\Gamma)$ if necessary,
we may assume that $c_0(\Gamma) \subset C(X)$.
Denote by $\pi \colon \beta\Gamma \rightarrow X$ the factor map associated to the inclusion $C(X) \subset C(\beta\Gamma)$.
Define
\[G^\ast(\Gamma):=\{ a\in C(\beta\Gamma)\rc \Gamma:E_s(a)\in c_0(\Gamma){\rm~for~all~}s\in \Gamma\}.\]
Note that this ideal corresponds to the ideal of all Ghost operators (\cite{RW}, Definition 1.2) in the uniform Roe algebra $C_u^\ast(\Gamma)$ under the canonical isomorphism $C_u^\ast(\Gamma) \cong C(\beta\Gamma)\rc \Gamma$ (see \cite{BO}, Remark 5.5.4).
Since $\Gamma$ is non-exact, by Theorem 5.5.7 of \cite{BO},
it does not have property A (with respect to a proper left invariant metric).
Therefore, by Remark 4.3 of \cite{RW},
$G^\ast(\Gamma)$ is non-separable.
Since $C(X) \rc \Gamma$ is separable, one can choose an element
\[a\in G^\ast(\Gamma) \setminus C(X) \rc \Gamma.\]
We define the intermediate \Cs-algebra $A$ of
$C(X)\rc \Gamma \subset C(\beta\Gamma)\rc \Gamma$ to be
\[A:=\Cso(a, C(X), \Gamma).\]
Then, by the choice of $a$,
we have $E(A) = C(X)$, $A \neq C(X) \rc \Gamma$.
Therefore $A$ is not of the form $C(Z) \rc \Gamma$ for any
intermediate extension $Z$ of $\pi$.

\end{proof}
\section{Intermediate von Neumann algebras and measurable dynamical systems}\label{Sec:W}
In this section, we prove the Main Theorem for von Neumann algebras.
The proof follows basically the same line as that of Theorem \ref{Thm:MainC}.
However, since the $\sigma$-strong topology is more delicate than the norm topology,
we need slightly more arguments.

As usual, we only consider von Neumann algebras with separable predual (though most results have an appropriate extension to the general case).
All inclusions of von Neumann algebras are assumed to be unital.
The basic facts and notations on von Neumann algebras can be found in \cite{Tak1}.
For a family $S_1, \ldots, S_n$ of subsets or elements in a von Neumann algebra $M$,
denote by $\Wso(S_1, \ldots, S_n)$ the von Neumann subalgebra of $M$ generated by
$S_1, \ldots, S_n$.
Let $\Gamma$ be a discrete group.
Similar to the \Cs-algebra case,
for a $\Gamma$-von Neumann algebra $M$,
denote by $E, E_s; s\in \Gamma$ the canonical conditional expectation
and the $s$-th coefficient map on $M \vnc \Gamma$ respectively.

We refer the reader to \cite{Gla} and \cite{Zim}
for basic facts and definitions of ergodic theory.
In measurable spaces, without stating,
\begin{itemize}
\item considered subsets are assumed to be measurable,
\item we ignore null sets
(for instance, the symbol `$A\subset B$' means that $A \setminus B$ is a null set),
\item all actions on a standard Borel space are assumed to be non-singular,
\item all factor maps $\pi \colon (Y, \nu) \rightarrow (X, \mu)$
are assumed to preserve the measure class
(that is, $\pi_\ast(\nu)$ and $\mu$ share null sets).
\end{itemize}

In the measurable setting, we define the notion of intermediate extensions and their equivalence analogous to the topological setting (see the beginning of Section\ \ref{Sec:C}).
In particular, for an extension $\pi \colon (Y, \nu) \rightarrow (X, \mu)$,
there is a bijective correspondence
between the set of intermediate $\Gamma$-von Neumann algebras
of $L^\infty(X) \subset L^\infty(Y)$
and the set of equivalence classes of intermediate extensions of $\pi$.
With this correspondence, we introduce the lattice structure on
the set of equivalence classes of intermediate extensions of $\pi$.

We first recall the following standard lemma.
\begin{Lem}\label{Lem:efree}
Let $\alpha \colon \Gamma \curvearrowright (X, \mu)$
be an essentially free action of a countable group $\Gamma$ on a standard probability space $(X, \mu)$.
Then, for any finite subset $F$ of $\Gamma\setminus \{e\}$,
one can choose a partition
$X= X_1 \sqcup \cdots \sqcup X_{3^{|F|}}$ of $X$
satisfying $X_i \cap \alpha_s(X_i)=\emptyset$ for all $s\in F$ and $i= 1, \ldots, 3^{|F|}$.
\end{Lem}
\begin{proof}
Obviously it suffices to show the claim for singletons $F$.
Let $s\in \Gamma \setminus \{e\}$ be given.
Set
\[\mathfrak{D}_s:= \{A \subset X: A \cap \alpha_s(A) = \emptyset\}.\]
Observe that by essential freeness,
any non-null set $A \subset X$
contains a non-null set in $\mathfrak{D}_s$.
We show that $\mathfrak{D}_s$ contains a maximal element
with respect to the inclusion order.
To see this, take any chain $\mathcal{C}$ in $\mathfrak{D}_s$.
Then there is an increasing sequence $(A_n)_n$ in $\mathcal{C}$
with $\lim_{n \rightarrow \infty} \mu(A_{n}) = \sup_{A \in \mathcal{C}} \mu(A)$.
Now $A = \bigcup_{n=1}^\infty A_{n}$
gives an upper bound of $\mathcal{C}$ in $\mathfrak{D}_s$.
By Zorn's lemma,
one can find a maximal element $A$ in $\mathfrak{D}_s$.
We claim that $\alpha_s(A) \cup A \cup \alpha_s^{-1}(A) =X$.
Indeed, if it fails, one can choose a non-null subset $B$
of $X \setminus (\alpha_s(A) \cup A \cup \alpha_s^{-1}(A))$ contained in $\mathfrak{D}_s$.
Then $B \cup A$ gives an element of $\mathfrak{D}_s$
larger than $A$, a contradiction.
Thus $\alpha_s(A) \cup A \cup \alpha_s^{-1}(A)=X$.
Now $X_1:=A, X_2 := \alpha_s(A) \setminus X_1, X_3 := \alpha_s^2(A) \setminus(X_1 \cup X_2)$
gives the desired partition.
\end{proof}

For a faithful normal state $\varphi$ on a von Neumann algebra $M$,
define the norm $\| \cdot \|_\varphi$ on $M$
by \[\| a\|_\varphi:= \varphi (a^\ast a)^{\frac{1}{2}}{\rm~for~}a\in M.\] Obviously $\| \cdot \|_\varphi$
is continuous with respect to the $\sigma$-strong topology.
Moreover $\| \cdot \|_\varphi$ induces the $\sigma$-strong topology
on any (operator norm) bounded subset of $M$.
This norm is convenient (and will be used) to estimate the $\sigma$-strong convergence of bounded nets in $M$.
The set
\[M_\varphi:= \{ a\in M: \varphi(ab)= \varphi(ba) {\rm\ for\ all\ } b\in M\}\]
 is called the centralizer of $\varphi$.
Obviously $M_\varphi$
is a finite von Neumann subalgebra of $M$.

Let $(a_i)_{i \in I}$ be a bounded family in $M$ with $a_i^\ast a_ j = a_i a_j^\ast=0$ for all $i \neq j$.
Then the series of $(a_i)_{i \in I}$ $\sigma$-strongly converges in $M$.
We denote it by $\sum_{i\in I} a_i$.
\begin{Lem}\label{Lem:ucpW}
Let $M$ be a von Neumann algebra with a faithful normal state $\varphi$.
Let $\Phi \colon M\rightarrow M$ be a normal unital completely positive map.
Let $P$ be a set of projections in $M_\varphi$.
Set
\[P^{\mathbb{N}, \perp}:=\left\{(p_i)_{i=1}^\infty \in P^{\mathbb{N}}: p_i \perp p_j {\rm~for~}i \neq j\right\}.\] 
For $a\in M$, define
\[\mathfrak{S}(a; P):=\stcl\left\{\sum_{i=1}^\infty p_i a p_i \colon (p_i)_{i=1}^\infty \in P^{\mathbb{N}, \perp}\right\}.\]
Then the set
\[\mathfrak{S}(\Phi; P):=\{ a\in M : \Phi(a) \in \mathfrak{S}(a; P)\}\]
 is $\sigma$-strongly closed.
\end{Lem}
\begin{proof}
We first observe that for any $(p_i)_{i=1}^\infty \in P^{\mathbb{N}, \perp}$,
the map $\Psi \colon M \rightarrow M$
defined by
$\Psi(a):= \sum_{i=1}^\infty p_i a p_i$
satisfies $\|\Psi(a)\|_\varphi \leq \|a\|_\varphi$.
Indeed, as $p_i \in M_\varphi$ for all $i$, \[\varphi(\Psi(a)^\ast \Psi (a))
=\varphi(\sum_{i, j} p_i a^\ast p_ip_j a p_j) 
=\varphi(\sum_{i} p_i a^\ast p_i^2 a p_i)
\leq \varphi(\sum_{i} a^\ast ap_i )
\leq \varphi(a^\ast a).\]

Let $a$ be taken from the $\sigma$-strong closure of
$\mathfrak{S}(\Phi; P)$.
Then for any $\epsilon>0$, one can choose
$b\in M$ and $(p_i)_{i=1}^\infty \in P^{\mathbb{N}, \perp}$
satisfying \[\|b -a\|_\varphi<\epsilon,\] \[\|\Phi(b)-\Phi(a)\|_\varphi< \epsilon,\]
\[\| \Phi(b)-\sum_{i=1}^\infty p_i b p_i \|_\varphi <\epsilon.\]
By the inequality proved in the previous paragraph, we have
\[\| \sum_{i=1}^\infty p_i a p_i - \sum_{i=1}^\infty p_i b p_i \|_\varphi \leq \| a-b\|_\varphi <\epsilon.\]
Combining these inequalities,
we obtain
\[\|\Phi(a) -\sum_{i=1}^\infty p_i a p_i\|_\varphi <3\epsilon.\]
As the set $\{\sum_{i=1}^\infty p_i a p_i: (p_i)_{i=1}^\infty \in P^{\mathbb{N}, \perp}\}$ is bounded,
this proves $ a \in \mathfrak{S}(\Phi; P)$ as desired. 
\end{proof}
We recall the following description of the commutant of the crossed product von Neumann algebra.
For a discrete group $\Gamma$, denote by $\rho$ the right regular representation of $\Gamma$ on $\ell^2(\Gamma)$.
\begin{Prop}[\cite{Tak1}, Chapter V, Proposition 7.14]\label{Prop:comm}
Let $M$ be a $\Gamma$-von Neumann algebra represented on a Hilbert space $H$.
Let $v$ be a unitary representation of $\Gamma$ on $H$
implementing the $\Gamma$-action on $M$.
Let $M \vnc \Gamma \rightarrow \mathcal{B}(H \otimes \ell^2(\Gamma))$
be the left regular representation.
Then the commutant of $M \vnc \Gamma$ in $\mathcal{B}(H \otimes \ell^2(\Gamma))$
is equal to $M' \vnc \Gamma$.
Here $\Gamma$ acts on $M'$ through $v$, and $M' \vnc \Gamma$ is represented on $H \otimes \ell^2(\Gamma)$
by the covariant representation $(\iota, \tilde{\rho})$
where
$\iota \colon M' \rightarrow \mathcal{B}(H)\otimes \mathbb{C}\id_{\ell^2(\Gamma)}$
is the amplified representation
and $\tilde{\rho}:= v \otimes \rho$.
\end{Prop}
The following result immediately follows from Proposition \ref{Prop:comm} and the bicommutant theorem.

\begin{Cor}\label{Cor:vnc}
Let $N\subset M$ be an inclusion of $\Gamma$-von Neumann algebras.
Then 
\[N \vnc \Gamma=\{ a\in M \vnc \Gamma : E_s(a)\in N{\rm ~for~all~}s\in \Gamma\}.\]
\end{Cor}
\begin{Rem}
The \Cs-algebra analogue of Corollary \ref{Cor:vnc}
holds true when the acting group has the AP (Proposition 3.4 of \cite{Suz17}).
However it fails for non-exact groups, even when the coefficients are commutative as shown in Section 5 of \cite{RW}. In the intermediate cases (\cite{LS}), nothing is known.
\end{Rem}
\begin{Thm}[The Main Theorem, \Ws-case]\label{Thm:MainW}
Let $\Gamma$ be a countable group.
Let $\alpha \colon \Gamma \curvearrowright (X, \mu)$, $\beta \colon \Gamma \curvearrowright (Y, \nu)$ be essentially free actions on a standard probability space.
Let $\pi \colon Y\rightarrow X$ be a factor map.
We regard $L^\infty(X)$ as a $\Gamma$-von Neumann subalgebra of $L^\infty(Y)$ via $\pi$.
Then the map 
\[L^\infty(Z) \mapsto L^\infty(Z)\vnc \Gamma \]
gives a lattice isomorphism between the lattice of intermediate extensions $Z$ of $\pi$
and that of intermediate von Neumann algebras
of 
$L^\infty(X)\vnc \Gamma \subset L^\infty(Y)\vnc \Gamma$.
\end{Thm}

\begin{proof}
Let $a\in L^\infty (Y) \ac \Gamma$ be given.
Set $F:=\{ s\in \Gamma\setminus \{e\} : E_s(a)\neq 0\}.$
By the choice of $a$, $F$ is finite.
Choose a partition $(X_i)_{i=1}^n$ of $X$
as in Lemma \ref{Lem:efree} for $F$.
Set $p_i:=\chi_{X_i}$ for each $i$.
Define $\varphi:= \nu \circ E$.
A direct computation shows that $L^\infty(Y) \subset (L^\infty(Y) \vnc \Gamma)_\varphi$.
Let $P(X)$ denote the set of projections in $L^\infty(X)$.
Then $(p_i)_{i=1}^n \in P(X)^{\mathbb{N}, \perp}$
and $\sum_{i=1}^n p_i =1$.
By the choice of $(X_i)_{i=1}^n$, for any $s\in F$, we have
\[\sum_{i=1}^n p_i u_s p_i
=\sum_{i=1}^n p_i \alpha_s(p_i) u_s=0.\]
As a result, we obtain
\[ \sum_{i=1}^n p_i a p_i =E(a).\]
Hence
\[L^\infty(Y) \ac \Gamma \subset \mathfrak{S}(E; P(X)).\]
Lemma \ref{Lem:ucpW} then implies
\[\mathfrak{S}(E; P(X))=L^\infty(Y) \vnc \Gamma.\]
In particular, for any intermediate von Neumann algebra
$Q$ of $L^\infty(X) \vnc \Gamma \subset L^\infty(Y) \vnc \Gamma$,
we have $E(Q) \subset Q$.
Corollary \ref{Cor:vnc} then yields
$Q= E(Q) \vnc \Gamma$.
\end{proof}
Theorem \ref{Thm:MainW} and the following significant theorem of Margulis \cite{Mar78} lead to
new natural and concrete examples of maximal amenable subalgebras.
It is remarkable that the lattices of their intermediate von Neumann algebras
are completely describable.
Some of them are extremal in the sense that
they admit no proper intermediate W$^\ast$-algebras.

For a locally compact second countable group $G$ and its closed subgroup $H$,
the homogeneous space $G/H$ admits a $G$-quasi-invariant probability
measure, which is unique up to measure class. We denote (one of) this measure by $m_{G/H}$, and regard $G/H$ as the standard probability space
$(G/H, m_{G/H})$.
We refer the reader to the book \cite{Kna} for basic facts on Lie groups.
In particular, for parabolic subgroups, see VII.7 of \cite{Kna}.
\begin{Thm}[Margulis' factor theorem \cite{Mar78}, cf.~Theorem 8.1.3 in \cite{Zim}]\label{Thm:Mar}
Let $G$ be a connected semisimple Lie group of real rank at least $2$ with no compact factors and trivial center.
Let $\Gamma$ be an irreducible lattice of $G$.
Let $P$ be a minimal parabolic subgroup of $G$.
Then the map
\[Q \mapsto L^\infty(G/Q)\]
gives a lattice anti-isomorphism from
the lattice of intermediate closed subgroups $Q$ of $P\subset G$
onto that of $\Gamma$-invariant von Neumann subalgebras
of $L^\infty(G/P)$.
\end{Thm}
\begin{Cor}\label{Cor:Mar}
Let $G$, $\Gamma$, $P$ be as in Theorem \ref{Thm:Mar}.
Let $Q$ be a proper intermediate closed group of $P \subset G$
with $\bigcap_{g\in G} gQg^{-1} \cap \Gamma= \{ 1\}$
$($this is automatic when $G$ is simple$)$.
Then the map
\[R\mapsto L^\infty(G/R)' \vnc\Gamma {\rm ~(where~the~ commutant~ is~ taken~ in~} \mathcal{B}(L^2(G/P)))\]
gives a lattice isomorphism between
the lattice of intermediate closed groups $R$ of $P\subset Q$
and that of intermediate von Neumann algebras of $L^\infty(G/P)\vnc \Gamma \subset L^\infty(G/Q)'\vnc \Gamma$.
In particular $L^\infty(G/P)\vnc \Gamma$ is a maximal amenable
subalgebra of the non-amenable factor $L^\infty(G/Q)' \vnc \Gamma $.
\end{Cor}
\begin{proof}
The assumption on $Q$ yields the essential freeness of
$\Gamma \curvearrowright G/Q$;
see Remark 13 of \cite{Oza16} for instance.
Theorems \ref{Thm:MainW}, \ref{Thm:Mar}, and the bicommutant theorem
confirm that the stated map indeed defines a lattice isomorphism.
We also recall that minimal parabolic subgroups of $G$ are maximal amenable in $G$.
In particular all intermediate closed groups of the inclusion $P \subset Q$ but $P$
are non-amenable.
By Corollary 4.3.7 of \cite{Zim}, for any closed subgroup $H$ of $G$,
the action $\Gamma \curvearrowright (G/H, m_{G/H})$
is amenable (see \cite{Zim}, Definition 4.3.1, or \cite{Ana87}) if and only if $H$ is amenable.
By Theorem 2.1 of \cite{Zim77}, Theorem \ref{Thm:MainW}, and the maximal
amenability of $P \subset G$, we conclude the maximal amenability of $L^\infty(G/P)\vnc \Gamma \subset L^\infty(G/Q)' \vnc \Gamma$. Factoriality of these crossed products follows from
Moore's ergodicity theorem (\cite{Zim}, Corollary 2.2.6).
\end{proof}
\begin{Rem}[Explicit presentation of $L^\infty(G/Q)'$]
Let $P \subset Q \subset G$ be closed inclusions of locally compact second countable groups.
Here we present $L^\infty(G/Q)'\subset \mathcal{B}(L^2(G/P))$
as the algebra of ``block-diagonal operators along the decomposition $G/P=\bigsqcup_{x\in G/Q} x/P$.''
To see this, fix a $Q$-quasi-invariant probability measure $\nu_e$ on $Q/P$.
Take a measurable cross section $s \colon G/Q \rightarrow G$ (\cite{Zim}, Corollary A.8).
Define
$\nu_x:= s(x)_\ast(\nu_e)$ for $x\in G/Q$.
Note that each $\nu_x$ is concentrated on $x/P$.
The integral
\[\int_{G/Q} \nu_{x} \,dm_{G/Q}\]
then defines a $G$-quasi-invariant measure on $G/P$.
We employ this integral for $m_{G/P}$.
Along this integral decomposition, we obtain the direct integral decomposition
\[L^2(G/P) = \int^{\oplus}_{G/Q}L^2(x/P, \nu_x) \,dm_{G/Q}.\]
(See Chapter IV.8 of \cite{Tak1} for direct integrals.)
This leads to the direct integral decomposition
\[L^\infty(G/Q)' = \int^{\oplus}_{G/Q}\left\{\mathcal{B}(L^2(x/P, \nu_x)), L^2(x/P, \nu_x)\right\} \,dm_{G/Q}.\]
The left translations define a left $G$-action on the right hand side,
and it describes the left $G$-action on $L^\infty(G/Q)'$ via the above identification.
\end{Rem}
\begin{Exm}\label{Exm:Mar}
Here we demonstrate a few explicit consequences of Corollary \ref{Cor:Mar}.
\begin{enumerate}
\item
Put $G=\SL(3, \mathbb{R})$, $\Gamma =\SL(3, \mathbb{Z})$,
\[P:= \{ [a_{i, j}]_{i, j}\in G: a_{i, j}=0 {\rm ~for~ all~} i<j\},\]
\[Q=\{ [a_{i, j}]_{i, j}\in G: a_{1, j}=0 {\rm ~for~} j=2, 3\}.\]
We then have natural identifications
\[G/P ={\rm Fl}_3(\mathbb{R})\ ({\rm the~full~flag~manifold}),\
G/Q = \mathbb{P}^2\mathbb{R},\]
given by \[AP \mapsto (A(\mathbb{R}\oplus 0_2), A(\mathbb{R}^2 \oplus 0_1), \mathbb{R}^3),\ AQ \mapsto A(\mathbb{R}\oplus 0_2) {\rm~for~}A \in G.\]
The factor map $G/P \rightarrow G/Q$ corresponds to the map
\[ [V_1, V_2, V_3] \mapsto [V_1].\]
A direct computation shows that the inclusion $P \subset Q$ has
no proper intermediate groups.
Hence, by Corollary \ref{Cor:Mar}, 
the amenable subfactor 
\[L^\infty({\rm Fl}_3(\mathbb{R})) \vnc \Gamma \subset L^\infty(\mathbb{P}^2\mathbb{R})' \vnc \Gamma\]
has no proper intermediate von Neumann algebras.
It is worth mentioning that Ozawa \cite{Oza16} recently showed that $L^\infty(\mathbb{P}^2\mathbb{R}) \vnc \Gamma $, hence $L^\infty(\mathbb{P}^2\mathbb{R})' \vnc \Gamma$, is a full factor.
\item
For $n \geq 3$, the following quadruplet $(G, \Gamma, P, R)$ satisfies the assumptions of Corollary \ref{Cor:Mar}.
\[G= \PSL(n, \mathbb{R}),\ \Gamma =\PSL(n, \mathbb{Z}),\]
\[P=\{ [a_{i, j}]_{i, j}\in G: a_{i, j}=0 {\rm ~for~ all~} i<j\},\]
\[R=\{ [a_{i, j}]_{i, j}\in G: a_{n, j}=0 {\rm ~for~all~} j<n \}.\]
We completely describe the lattice $\mathbb{L}$ of intermediate von Neumann algebras of
\[N:=L^\infty(G/P) \vnc \Gamma \subset L^\infty(G/R)' \vnc \Gamma:=M.\]
For $S \subset \{1, \ldots, n-1 \}$, define
\[G_S:=\left\{A\in G: A (\mathbb{R}^{d} \oplus 0_{n-d})= \mathbb{R}^{d} \oplus 0_{n-d} {\rm ~for~all~} d\in \{1, \ldots, n-1\} \setminus S\right\}.\]
Denote by $\mathfrak{P}(n-2)$ the lattice of subsets of $\{1, \ldots, n-2\}$.
Then the map $S \mapsto G_S$
gives a lattice isomorphism from $\mathfrak{P}(n-2)$ onto the lattice of intermediate subgroups of $P \subset R$ (see \cite{Kna}, VII.7 Examples (1)).
Consequently the map
\[S \mapsto L^\infty(G/ G_S)'\vnc \Gamma\]
gives a lattice isomorphism
from $\mathfrak{P}(n-2)$ onto $\mathbb{L}$. By Theorem 11 of \cite{Oza16},
all intermediate von Neumann subalgebras of $N\subset M$ but $N$ are full.
Note that $G/G_S$ is identified with the partial flag manifold of signature $t_1, \ldots, t_r$
where $t_1< \cdots< t_r$ denote the elements of $\{1, \ldots, n\} \setminus S$.
\item
Let $G, P, R, G_S$ be as in $(2)$.
Set $\mathcal{O}:=\mathbb{Z}[\sqrt{2}]$.
Let $\sigma$ be the ring automorphism on $\mathcal{O}$
given by $\sigma(\sqrt{2})=-\sqrt{2}$.
Define \[\Lambda := \{(A, \sigma(A)): A\in \PSL(n, \mathcal{O})\}.\]
Then the quadruplet $(G^2, \Lambda, P^2, G \times R)$
satisfies the assumptions of Corollary \ref{Cor:Mar} (see Example 2.2.5 of \cite{Zim}).
Denote by $\mathbb{M}$ the lattice of intermediate von Neumann algebras of
\[L^\infty(G^2/P^2) \vnc \Lambda \subset L^\infty(G^2/ G \times R)' \vnc \Lambda\] 
Put $G_{S, T}:= G_{S}\times G_{T}$ for $(S, T)\in \mathfrak{P}(n-1) \times \mathfrak{P}(n-2)$.
Then the map
\[(S, T) \mapsto L^\infty(G^2/ G_{S, T} )'\vnc \Lambda\]
gives a lattice isomorphism from $\mathfrak{P}(n-1) \times \mathfrak{P}(n-2)$
onto $\mathbb{M}$.
\end{enumerate}
\end{Exm}
We end this section by giving a remark on non-free actions.
Similar to the \Cs-algebra case (Proposition \ref{Prop:topfree}), Theorem \ref{Thm:MainW} fails for non-free actions.
\begin{Prop}\label{Prop:nonfree}
Let $\Gamma$ be a countable group.
Let $\alpha \colon \Gamma \curvearrowright (X, \mu)$ be a non-essentially free action
on a standard probability space.
Then there is an essentially free action $\beta \colon \Gamma \curvearrowright (Y, \nu)$ on a standard probability space
and a factor map $\pi \colon Y \rightarrow X$ with the following property.
There is an intermediate von Neumann algebra $M$ of $L^\infty(X) \vnc \Gamma \subset L^\infty(Y) \vnc \Gamma$ not of the form $L^\infty(Z) \vnc \Gamma$
for any intermediate extension $Z$ of $\pi$.
\end{Prop}
\begin{proof}
The proof is basically the same as that of Proposition \ref{Prop:topfree}.
However, since single points no longer make sense in the measurable setting,
we need a slight modification. 

Take $s\in \Gamma \setminus \{e\}$
which trivially acts on a non-null set $U \subset X$.
Let $\Lambda$ be the subgroup of $\Gamma$ generated by $s$.
Define $p \colon \Gamma / \Lambda \times U \rightarrow \Gamma U$
by $p(s\Lambda, u):= su$.
We equip $\Gamma / \Lambda \times U$ with the $\Gamma$-action
$s(t\Lambda, u):=(st\Lambda, u)$; $s, t \in \Gamma$, $u \in U$.
Then $p$ is a factor map.
Define a unital normal embedding \[\iota \colon L^\infty(X) \rightarrow L^\infty(X) \oplus L^\infty(\Gamma / \Lambda \times U)\] to be
$\iota(f) := (f, f \circ p)$.
Let $L^\infty(Y_0)$ denote the von Neumann algebra generated by
$\iota( L^\infty(X))$ and $0 \oplus L^\infty(\Gamma / \Lambda \times U)$.
We regard $L^\infty(X)$ as a (unital) von Neumann subalgebra of $L^\infty(Y_0)$
via $\iota$. We also identify $L^\infty(\Gamma / \Lambda \times U)$ with
$0 \oplus L^\infty(\Gamma / \Lambda \times U)$.
Now consider the diagonal action
$\beta \colon \Gamma \curvearrowright Y:= \Gamma \times Y_0$.
This is a free action on the standard probability space $Y$.
Let $q\colon Y \rightarrow Y_0$ be the projection onto the second coordinate.
Define $\pi:= p \circ q \colon Y \rightarrow X$.

We show that $\pi$ is the desired extension.
To show this, observe that, since $\ell^\infty(\Gamma/\Lambda) \vnc \Gamma \cong \mathcal{B}(\ell^2(\Gamma/\Lambda)) \votimes L(\Lambda)$ is not a factor,
one can choose a weakly closed ideal $I$
of $L^\infty(\Gamma / \Lambda \times U) \vnc \Gamma$
satisfying $I \cap L^\infty(\Gamma / \Lambda \times U)= 0$, $E(I) =L^\infty(\Gamma / \Lambda \times U)$.
Then, similar to the proof of Proposition \ref{Prop:topfree},
\[M:= L^\infty(X) \vnc \Gamma + I\] gives the desired intermediate
von Neumann algebra.
\end{proof}
\section{Non-commutative dynamical systems}\label{Sec:NC}
In this section we extend the Main Theorem to non-commutative dynamical systems. We then give some applications,
including the lattice realization theorem Corollary \ref{Corint:Lat}.

To formulate a non-commutative variant of the Main Theorem, we first
need to introduce a few definitions.
For a \Cs-algebra $A$, define
\[\ell^2(A):= \left\{ a= (a_i)_{i=1}^\infty \in A^\mathbb{N}:
{\rm the\ series\ }\sum_{i=1}^\infty a_i^\ast a_i{\rm \ converges\ in \ norm}\right\}.\]
We equip $\ell^2(A)$ with the norm $\| a \|_2 := \|\sum_{i=1}^\infty a_i^\ast a_i\|^{1/2}$.
When $A$ is a $\Gamma$-\Cs-algebra,
we equip $\ell^2(A)$ with the pointwise $\Gamma$-action.
For $a= (a_i)_{i=1}^\infty, b= (b_i)_{i=1}^\infty \in \ell^2(A)$, $c\in A$,
we define
\[a^\ast b := \sum_{i=1}^\infty a_i^\ast b_i \in A,\]
\[ac:= (a_i c)_{i=1}^\infty,~ ca:=(c a_i)_{i=1}^\infty \in \ell^2(A).\]
It is not hard to check
\[\| a^\ast b\| \leq \| a\|_2 \| b \|_2,\
\max \{\|ac\|_2, \|ca\|_2\} \leq \|a\|_2 \|c \|.\]
\begin{Def}\label{Def:cfC}
Let $A \subset B$ be an inclusion of $\Gamma$-\Cs-algebras.
We say that the inclusion is {\it centrally $\Gamma$-free}
if for any $b, c\in B$, $\epsilon>0$, $s\in \Gamma \setminus \{e \}$,
there is an element $a\in \ell^2( A)_1 (\subset \ell^2(B))$
satisfying
$\|a^\ast b a - b\| <\epsilon$,
$\|a^\ast c s(a)\| < \epsilon$.
\end{Def}
For a $\Gamma$-\Cs-algebra $A$,
we say that $A$
is {\it centrally $\Gamma$-free},
or say that the underlying $\Gamma$-action is {\it centrally free},
if the identity inclusion $A \subset A$
is centrally $\Gamma$-free.
Note that for commutative \Cs-algebras,
central freeness is equivalent to the freeness of the action on the Gelfand spectrum.
\begin{Rem}
Unlike the measurable case (Lemma \ref{Lem:efree}),
we have no universal bound of the cardinals of open covers in Lemma \ref{Lem:freet}.
(The situation remains the same in the non-commutative case. See Remark \ref{Rem:cf}.)
This is the reason why we adopt $\ell^2(A)$
in the formulation of central freeness.
Its importance can be read from Examples \ref{Exm:NCBC} and \ref{Exm:NCC}. 
\end{Rem}
For von Neumann algebras,
it is convenient to use the notion of ultraproduct.
Throughout this section, we fix
a free ultrafilter $\omega$ on $\mathbb{N}$.
Here we recall a few definitions.
For complete treatments, we refer the reader to \cite{AH}.
Let $M$ be a von Neumann algebra.
Let $\ell^\infty(M)$ denote the \Cs-algebra of bounded sequences in $M$.
Set
\[\mathcal{I}_\omega(M) := \left\{ (a_i)_{i=1}^\infty \in \ell^\infty(M):
\sslim\lim_{i\rightarrow \omega} a_i= 0\right\},\]
\[ \mathcal{M}^\omega(M):= \left\{ a\in \ell^\infty(M): a \mathcal{I}_\omega(M),
\mathcal{I}_\omega(M) a \subset \mathcal{I}_\omega(M) \right\},\]
\[\mathcal{M}_\omega(M):= \left \{ (a_i)_{i=1}^\infty \in \ell^\infty(M): \lim_{i \rightarrow \omega}\| \varphi(a_i ~\cdot~) - \varphi (~\cdot~ a_i)\|=0 {\rm \ for \ all \ }\varphi \in M_\ast \right\}.\]
The quotient \Cs-algebras $M^\omega:= \mathcal{M}^\omega(M)/\mathcal{I}_\omega(M)$, $M_\omega:=\mathcal{M}_\omega(M)/ \mathcal{I}_\omega(M)$,
are called Ocneanu's ultraproduct of $M$ and
Connes' asymptotic centralizer of $M$ respectively.
It is known that $M^\omega$ is a von Neumann algebra (possibly with non-separable predual).
Moreover $M_\omega$ sits in $M^\omega$ as a von Neumann subalgebra.
For any faithful normal state $\varphi$ on $M$,
the map $(a_i)_{i=1}^\infty \in \ell^\infty(M) \mapsto \lim_{i\rightarrow \omega} \varphi(a_i)$ induces
a faithful normal state $\varphi^\omega$ on $M^\omega$
with $M_\omega \subset (M^\omega)_{\varphi^\omega}$.
We identify $M$ with a von Neumann subalgebra of $M^\omega$
via the map $ a \mapsto (a)_{i=1}^\infty + \mathcal{I}_\omega(M)$.
For a $\Gamma$-von Neumann algebra $M$,
we equip $M^\omega$ with the $\Gamma$-action induced from the pointwise action.

Next consider an inclusion $N \subset M$ of von Neumann algebras.
When it admits a faithful normal conditional expectation,
we have $N^\omega \subset M ^\omega$.
In general, this is not true.
We thus define 
\[N^\omega \cap M_\omega := [\ell^\infty(N) \cap \mathcal{M}_\omega(M)]/\mathcal{I}_\omega(N).\]
It is not hard to show that $N^\omega \cap M_\omega$ is a von Neumann subalgebra of $M_\omega$.
Indeed, this follows from a slight modification of the proof of
the fact that $M_\omega$ is a von Neumann algebra.
\begin{Def}\label{Def:cfW}
Let $N \subset M$ be an inclusion of $\Gamma$-von Neumann algebras.
We say that the inclusion is {\it centrally $\Gamma$-free}
if for any $s\in \Gamma \setminus \{e\}$,
there is a sequence $(p_i)_{i=1}^\infty$ of projections in $N^\omega \cap M_\omega$
satisfying
$\sum_{i=1}^\infty p_i =1$ and
$p_j s(p_j) =0$ for all $j\in \mathbb{N}$.
\end{Def}

\begin{Rem}\label{Rem:cf}
Here we list a few remarks related to Definition \ref{Def:cfW}.
\begin{enumerate}
\item
By Theorem 1.6 and Lemma 2.6 in Chapter XVII of \cite{Tak3},
Lemma \ref{Lem:efree}, and standard reindexation arguments,
the central $\Gamma$-freeness
of $N \subset M$
is equivalent to the proper outerness
of $\Gamma \curvearrowright N^\omega \cap M_\omega$.
In particular, the central freeness of a group action on a von Neumann algebra $M$ (\cite{Oc}, Section 5.2)
is equivalent to the central $\Gamma$-freeness of the identity inclusion $M\subset M$.
\item By the same reason, one can replace $\infty$ in Definition \ref{Def:cfW}
by $3$.
\item When $M$ is finite, $N^\omega \cap M_\omega = N^\omega \cap M'$.
Hence our notion of central freeness is compatible with
central triviality for automorphisms of type II$_1$ subfactors \cite{Kaw93}.
\end{enumerate}
\end{Rem}
We now state non-commutative analogues of the Main Theorem.
\begin{Thm}\label{Thm:NCMainC}
Let $\Gamma$ be a discrete group with the AP.
Let $A \subset B$ be an inclusion of $\Gamma$-\Cs-algebras.
Assume that the inclusion is centrally $\Gamma$-free.
Then the map 
\[D \mapsto D \rc \Gamma \]
gives a lattice isomorphism between the lattice of intermediate $\Gamma$-\Cs-algebras $D$ of $A\subset B$
and that of intermediate \Cs-algebras
of 
$A\rc \Gamma \subset B\rc \Gamma$.
\end{Thm}
\begin{Thm}\label{Thm:NCMainW}
Let $\Gamma$ be a countable discrete group.
Let $N \subset M$ be an inclusion of $\Gamma$-von Neumann algebras.
Assume that the inclusion is centrally $\Gamma$-free.
Then the map
\[P \mapsto P\vnc \Gamma \]
gives a lattice isomorphism between the lattice of intermediate $\Gamma$-von Neumann algebras $P$ of $N \subset M$
and that of intermediate von Neumann algebras
of 
$N\vnc \Gamma \subset M \vnc \Gamma$.
\end{Thm}
Theorems \ref{Thm:NCMainC} and \ref{Thm:NCMainW}
are proved in a similar way to the Main Theorem.
\begin{proof}[Proof of Theorem \ref{Thm:NCMainC}]
Similar to the proof of Theorem \ref{Thm:MainC},
it suffices to show $B \ac \Gamma \subset \mathfrak{N}(E; A)$.
Let $a\in B \ac \Gamma$ be given.
Set $F:=\{ s\in \Gamma \setminus \{e\}: E_s(a)\neq 0\}$.
Put $k:=|F|$.
Enumerate $F$ as $s_1, \ldots, s_k$.
We will show $a\in \mathfrak{N}(E; A)$ by induction on $k$.
The case $k=0$ is trivial.
Suppose we have shown the claim for $k-1$.
Let $\epsilon>0$ be given. Put $a_k := E_{s_k}(a)$.
Then by applying the induction hypothesis to $a - a_k u_{s_k}$,
one can choose $b \in \ell^2( A)_1$
satisfying
\[\|b^\ast (a -a_{k}u_{s_k})b -E(a)\| <\epsilon.\]
Since $A \subset B$ is centrally $\Gamma$-free,
one can choose $c \in \ell^2(A)_1$
satisfying 
\[\|c^\ast E(a) c - E(a) \| <\epsilon,\
\|c^\ast (b^\ast a_k s_k(b)) s_k(c)\| < \epsilon.\]
Combining these three inequalities, we obtain
\begin{eqnarray*}
&&\| c^\ast (b ^\ast a b) c -E(a)\|\\
&&\leq \|c^\ast( b^\ast (a - a_k u_{s_k}) b - E(a))c\| + \|c^\ast E(a) c -E(a) \|
+ \|c^\ast (b^\ast a_k s_k(b)) s_k(c)\|\\
&&< 3 \epsilon.
\end{eqnarray*}
Since $\epsilon>0$ is arbitrary, we conclude $a \in \mathfrak{N}(E; A)$.
\end{proof}
\begin{Lem}\label{Lem:finite}
Let $N \subset M$ be an inclusion of $\Gamma$-von Neumann algebras.
Assume that $N \subset M$ is centrally $\Gamma$-free.
Then for any finite subset $F \subset \Gamma \setminus \{e\}$,
there is a sequence $(p_i)_{i=1}^\infty$ of projections in $N^\omega \cap M_\omega$
satisfying
$\sum_{i=1} ^\infty p_i=1$
and $p_js(p_j) =0$ for all $s \in F$ and $j \in \mathbb{N}$.
\end{Lem}

\begin{proof}
This follows from standard reindexation arguments.
\end{proof}
\begin{Lem}\label{Lem:ultraW}
Let $N \subset M$ be an inclusion of $\Gamma$-von Neumann algebras.
Let $P_{N^\omega \cap M_\omega} \subset (M \vnc \Gamma)^\omega$ denote the set of projections in $N^\omega \cap M_\omega$.
Then for any $a\in M \vnc \Gamma$, we have
\[\mathfrak{S}(a; P_{N^\omega \cap M_\omega})\cap (M \vnc \Gamma) \subset \Wso(N, a).\]
\end{Lem}
\begin{proof}
Fix a faithful normal state $\varphi$ on $M$.
Put $\psi := \varphi \circ E$.
Let $b \in \mathfrak{S}(a; P_{N^\omega \cap M_\omega})\cap (M \vnc \Gamma)$ be given.
Then, by assumption,
for any $\epsilon >0$,
there are pairwise orthogonal projections $p_1, \ldots, p_k \in P_{N^\omega \cap M_\omega}$
satisfying $\| \sum_{i=1}^k p_i a p_i - b\|_{\psi^\omega} < \epsilon$.
Then, by a standard functional calculus argument, one can choose pairwise orthogonal projections $q_1, \ldots, q_k$ in $N$
with $\|\sum_{i=1}^k q_i a q_i -b\|_\psi <\epsilon$.
Since $\epsilon>0$ is arbitrary,
we conclude $b\in \Wso(N, a)$.
\end{proof}
\begin{proof}[Proof of Theorem \ref{Thm:NCMainW}]
We continue the notations of Lemma \ref{Lem:ultraW}.
Note that $P_{N^\omega \cap M_\omega} \subset (M^\omega)_{\varphi^\omega}$.
Hence, by Lemmas \ref{Lem:ucpW} and \ref{Lem:ultraW}, it suffices to show
$E(a) \in \mathfrak{S}(a; P_{N^\omega \cap M_\omega})$
for any $a\in M \ac \Gamma$.
(We remark that Lemma \ref{Lem:ucpW} is valid without separability of the predual.)

Let $a\in M \ac \Gamma$ be given.
Set $F:= \{s\in \Gamma \setminus \{e \}: E_s(a)\neq 0\}$.
Then by Lemma \ref{Lem:finite},
there is a sequence $(p_i)_{i=1}^\infty$
of projections in $N^\omega \cap M_\omega$
satisfying $\sum_{i=1}^\infty p_i=1$, $p_j s(p_j)=0$ for all $j\in \mathbb{N}$ and $s\in F$.
These relations yield
$\sum_{i=1}^\infty p_i a p_i =E(a)$ as desired.
\end{proof}
In the rest of this section,
we give examples of centrally $\Gamma$-free inclusions.
We first record the following permanence property of central freeness.
\begin{Rem}[Stability under tensor products and quotients]\label{Rem:stab}
Given a centrally $\Gamma$-free inclusion $A\subset B$ of $\Gamma$-\Cs-algebras.
Then, for any non-degenerate inclusion $C\subset D$ of $\Gamma$-\Cs-algebras and for any $\Gamma$-ideal $I$ of the maximal tensor product $B \otimes_{\rm max} D$,
the inclusion \[\Cso( A \cdot C, I)/I \subset (B \motimes D)/{I}\]
is again centrally $\Gamma$-free.
\end{Rem}

\begin{Exm}[Non-commutative Bernoulli shifts for \Cs-algebras]\label{Exm:NCBC}
Let $A \subset B$ be a unital inclusion of \Cs-algebras
and let $\Gamma$ be an infinite group. 
Assume that $A$ is simple (and non-commutative).
We will show that the inclusion of the non-commutative Bernoulli shifts
$\bigotimes_\Gamma A \subset \bigotimes _\Gamma B$ is centrally $\Gamma$-free.
Let $\sigma, \varrho \colon \Gamma \curvearrowright \bigotimes_\Gamma B$
denote the left and right shift action.
Let $s\in \Gamma \setminus \{e\}$ be given.
We first observe that there is a nonzero positive element $a\in \bigotimes_\Gamma A$
with $a \sigma_s(a)=0$.
To see this, for $t\in \Gamma$,
let $\iota_t \colon A  \rightarrow \bigotimes_\Gamma A$
denote the canonical embedding into the $t$-th tensor product factor.
Choose two nonzero positive elements $a_1, a_2 \in A$
with $a_1 a_2=0$.
Then $a:= \iota_e(a_1)\iota_s(a_2)$ possesses the desired property.
Now, since $\bigotimes_\Gamma A$ is unital and simple,
one can choose a sequence $x_1, \ldots, x_n \in A$ satisfying
$\sum_{i=1}^n x_i^\ast a^2 x_i=1$.
Set $x_i:=0$ for $i>n$ and
define $c:=(ax_i)_{i=1}^\infty \in \ell^2(A)$.
Then it is clear that $c^\ast c=1$, $c^\ast \sigma_s(c)=0$.
Now for any given $b_1, b_2$ in $\bigotimes_\Gamma B$ and $\epsilon>0$, choose $t\in \Gamma$ satisfying, with 
$d:= \varrho_t(c)$,
$\|db_i^\ast - b_i^\ast d\|< \epsilon$ for $i=1, 2$.
As $\varrho_t$ commutes with $\sigma_s$,
we have $d^\ast \sigma_s(d)= \varrho_t(c^\ast \sigma_s(c))=0$.
Hence \[\| d^\ast b_1 d -b_1\| \leq \|db_1^\ast -b_1^\ast d \|_2< \epsilon,\]
\[\| d^\ast b_2 \sigma_s(d)\| \leq \|db_2^\ast -b_2^\ast d \|_2 + \| b_2 \| \| d^\ast \sigma_s(d)\|
<  \epsilon.\]
This proves the central $\Gamma$-freeness of $\bigotimes_\Gamma A \subset \bigotimes _\Gamma B$.
\end{Exm}

\begin{Exm}[Infinite tensor product actions]\label{Exm:NCC}
Let $I$ be an infinite set.
For each $i \in I$,
let $B_i$ be a unital $\Gamma$-\Cs-algebra, $A_i$ be a unital $\Gamma$-\Cs-subalgebra of $B_i$, and assume that infinitely many $A_i$ admits a unital
centrally $\Gamma$-free \Cs-subalgebra $C_i$.
Put $A:= \bigotimes_{i\in I} A_i$, $B:= \bigotimes_{i\in I} B_i$.
We equip $A, B$, with the diagonal $\Gamma$-action.
Then the inclusion $A \subset B$ is centrally $\Gamma$-free.
\end{Exm}
\begin{Exm}[Minimal ambient nuclear \Cs-algebras]\label{Exm:minnuc}
We now construct new examples of minimal ambient nuclear \Cs-algebras (see \cite{SuzMin} for the first examples).
We emphasize that, in contrast to the von Neumann algebra case, the existence of a minimal ambient nuclear \Cs-algebra is not clear.
Indeed, the class of nuclear \Cs-algebras does not form a monotone class \cite{Suz17}.
Novelties of this approach are
\begin{enumerate}[(i)]
\item we only use the primeness of the action rather than the much stronger property called property $\mathcal{R}$
(see \cite{SuzMin}, Proposition 3.3),
\item we do not use the Powers property of (subgroups of) the acting group.
\end{enumerate}
(Here we recall that a topological dynamical system
$\alpha$ is said to be prime if it has no non-trivial proper factors.)

Let $\Gamma$ be a discrete group with the AP.
Take an amenable prime action $\alpha \colon \Gamma \curvearrowright X$
on a compact space $X$
(see Propositions 3.3 and 3.7 of \cite{SuzMin} for existence results).

Take a simple nuclear centrally free $\Gamma$-\Cs-algebra $A$ (cf.~ Examples \ref{Exm:NCBC}, \ref{Exm:NCC}).
Theorem 4.3.4 in \cite{BO} implies nuclearity of $[A \otimes C(X)] \rc \Gamma$.
By Theorem \ref{Thm:NCMainC} and the tensor splitting theorem \cite{Zac}, \cite{Zsi}, the inclusion $A \rc \Gamma \subset [A\otimes C(X)] \rc \Gamma$ has no proper intermediate \Cs-algebras.
Note that when $\Gamma$
is non-amenable, typically $A\rc \Gamma$ is non-nuclear (for instance when $A$ has a $\Gamma$-invariant state).
(However it is occasionally nuclear \cite{Suz18}.)
\end{Exm}
\begin{Exm}[Non-commutative Bernoulli shifts for von Neumann algebras]\label{Exm:NCBW}
Let $Q$ be a von Neumann algebra
with a faithful normal state $\varphi$
satisfying $Q_\varphi \neq \mathbb{C}$.
Here we only consider infinite groups $\Gamma$.
Let $\sigma$ and $\varrho$ denote
the left and right shift action of $\Gamma$ on $(M, \psi):=\bigvotimes_\Gamma(Q, \varphi)$ respectively.
Then $\sigma$ is centrally free.
To see this, choose an abelian von Neumann subalgebra
$\mathbb{C} \neq A \subset Q_\varphi$.
Then the action $\Gamma \curvearrowright B:=\bigvotimes_{\Gamma} (A, \varphi|_A)$
comes from a Bernoulli shift on a non-trivial standard probability space, which is essentially free.
Thus, by Lemma \ref{Lem:efree}, for each $s\in \Gamma\setminus \{e\}$,
one can choose projections $p_1, p_2, p_3$ in $B$
satisfying $p_1 + p_2 + p_3=1$, $p_j \sigma_s(p_j)=0$ for each $j$.
Take a sequence $(s_i)_{i=1}^\infty$ in $\Gamma$ tending to infinity.
The sequences
$(\varrho_{s_i}(p_j))_{i=1}^\infty$; $j=1, 2, 3$,
then define the projections $q_1, q_2, q_3 \in M_{\omega}$
satisfying $q_1 + q_2 +q_3=1$ and $q_j s(q_j)=0$ for all $j$.

\end{Exm}
\begin{Exm}[Twisted variant of Ge--Kadison's splitting theorem]
For any $\Gamma$-von Neumann algebra $N$
and for any centrally free $\Gamma$-factor $M$,
by applying Theorem \ref{Thm:NCMainW}
and \cite{GK} to $ M \subset M \votimes N$,
we obtain the lattice isomorphism
\[P \mapsto [M \votimes P]\vnc \Gamma.\]
between the lattice of $\Gamma$-von Neumann subalgebras $P$ of $N$
and that of intermediate von Neumann algebras of
$M \vnc \Gamma \subset [M \votimes N]\vnc \Gamma.$
This phenomenon can be seen as a twisted version of Ge--Kadison's
tensor splitting theorem \cite{GK}.
Combining this with Corollary \ref{Corint:Mar} (taking $M$ to be amenable),
we obtain further examples of maximal amenable subalgebras.

Now consider an amenable group $\Gamma$.
Let $M$ be a centrally $\Gamma$-free AFD II$_1$ factor.
Then, each measure preserving (not necessary free) prime action $\alpha$ of $\Gamma$ on a diffuse probability space $(X, \mu)$ (e.g.,~the Cha\`{c}on system (\cite{Gla}, Theorem 16.6)) associates an infinite index subfactor of the AFD II$_1$ factor with no proper intermediate von Neumann algebras.
Indeed, by the primeness of $\alpha$ and the above observation, the inclusion
$M \vnc \Gamma \subset [M \votimes L^\infty(X)] \vnc \Gamma$
has no proper intermediate von Neumann algebras.
\end{Exm}
\begin{Exm}[Realizations of the intermediate group lattices]\label{Exm:Lat}
Let $G$ be a locally compact second countable group and $H$ be a closed subgroup of $G$.
We realize the lattice of intermediate closed groups of $H \subset G$
as the lattice of intermediate subfactors of an irreducible subfactor.
(Recall that a subfactor $N \subset M$ is said to be irreducible
if all intermediate von Neumann algebras are factors.)
To see this, take a countable dense subgroup $\Gamma \subset G$.
We regard $\Gamma$ as a discrete group.
As $\Gamma$ is dense in $G$,
any $\Gamma$-invariant von Neumann subalgebra of $L^\infty(G/H)$ is
$G$-invariant. Hence it must be of the form $L^\infty(G/L)$
for some closed intermediate subgroup $L$ of $H \subset G$ (cf.~\cite{Zim}, Appendix B).
Now take a centrally free $\Gamma$-factor $N$ (see Example \ref{Exm:NCBW}).
As $[N \votimes L^\infty(G/L)'] \vnc \Gamma$ is a factor for any $L$
(Chapter V.7 of \cite{Tak1}), the inclusion
\[[N \votimes L^\infty(G/H)]\vnc \Gamma \subset [N\votimes \mathcal{B}(L^2(G/H))] \vnc \Gamma\]
possesses the desired properties.
Note that thanks to Theorem 2 of \cite{TT}, this result can be extended to general locally compact groups.
However, if the group is not second countable, then the factors are no longer separable.
\end{Exm}
Here we record the following interesting observation on the Galois correspondence theorem,
which is an immediate consequence of Example \ref{Exm:Lat}.
Here we recall that an action of a locally compact group $G$ on a factor $M$
is said to be minimal if the fixed point subalgebra $M^G$ of $G$
is irreducible in $M$.
\begin{Cor}
Any locally compact group $G$ admits a minimal action on a factor $M$ with the following property.
The map $H \mapsto M^H$
gives a lattice anti-isomorphism between the lattice of closed subgroups $H$ of $G$ and
the lattice of intermediate von Neumann algebras of $M^G \subset M$.
\end{Cor}
\begin{proof}
We use the notations used in Example \ref{Exm:Lat} with $H=\{ e\}$.
Set $M:=[N\votimes \mathcal{B}(L^2(G))]\vnc \Gamma$.
Consider the $G$-action
\[\id_N \votimes \rho \colon G \curvearrowright N\votimes \mathcal{B}(L^2(G)),\]
where $\rho \colon G \curvearrowright  \mathcal{B}(L^2(G))$ denotes the action induced from the right regular action. 
Then this action obviously commutes with the equipped $\Gamma$-action.
Hence it extends to the $G$-action on $M$. 
By the argument in Example \ref{Exm:Lat} (cf.~\cite{TT}), this action satisfies the required properties.
\end{proof}
\begin{Rem}
Related to Remark \ref{Rem:stab} and Examples \ref{Exm:minnuc} to \ref{Exm:Lat}, we remark that the equivariant version of the tensor splitting theorems \cite{GK}, \cite{Zac}, \cite{Zsi} fail in general. 
More precisely, for a $\Gamma$-simple \Cs-algebra $A$ (i.e., with no proper $\Gamma$-ideals) and for a unital $\Gamma$-\Cs-algebra
$B$, an intermediate $\Gamma$-\Cs-algebra of $A\subset A \otimes B$
does not need to split into $A \otimes B_0$, $B_0 \subset B$, and similarly for the von Neumann algebra case.
Indeed, let $\Gamma$ be a finite group of order at least $3$.
Consider the inclusion 
$\ell^\infty(\Gamma) \otimes \mathbb{C} \subset \ell^\infty(\Gamma)\otimes \ell^\infty(\Gamma)$.
Here $\Gamma$ is equipped with the left translation $\Gamma$-action.
Then $\ell^\infty(\Gamma)$ is $\Gamma$-simple.
On $\Gamma \times \Gamma$,
we define the equivalence relation $\sim$ by declaring
$(s, t)\sim (s', t')$ if $s= s'$
and either $t, t' \in \Gamma \setminus \{s\}$ or $t= t'$ holds.
Define $Z:= (\Gamma \times \Gamma) /\sim$.
Then $\ell^\infty(Z)$ defines an intermediate $\Gamma$-\Cs-algebra
of the inclusion $\ell^\infty(\Gamma)\otimes \mathbb{C} \subset \ell^\infty(\Gamma)\otimes \ell^\infty(\Gamma)$,
while $\ell^\infty(Z)$ does not split into $\ell^\infty(\Gamma) \otimes \ell^\infty(W)$
for any factor $W$ of $\Gamma$.
\end{Rem}

\section{Constructions of exotic endomorphisms}\label{Sec:Pf}
In this final section we construct
``exotic'' endomorphisms on \Cs-algebras.
Thanks to Kirchberg's tensor absorption theorem
(\cite{Kir94}, Theorem 3.15 of \cite{KP}), the next theorem implies Corollary \ref{Corint:End}.
The key idea is to realize
the Cuntz algebra $\mathcal{O}_\infty$
as a corner of the crossed product of an appropriate free group action.
This realization result itself would be of independent interest.

For the basic facts and notations on K-theory and KK-theory,
we refer the reader to \cite{Bla}.
\begin{Thm}\label{Thm:End}
Let $A$ be a simple \Cs-algebra satisfying $ A \otimes \mathcal{O}_\infty \cong A$.
Then there is an endomorphism $\sigma \colon A \rightarrow A$ with the following properties.
\begin{itemize}
\item
The inclusion $\sigma(A) \subset A$ has no proper intermediate \Cs-algebras.
\item
The inclusion $\sigma(A) \subset A$ does not admit a conditional expectation.
\end{itemize}
\end{Thm}
\begin{proof}

Let $\Gamma$
be a countable free group of infinite rank.
Take an amenable minimal topologically free action $\alpha$ of $\Lambda = \mathbb{F}_2$
on the circle $\mathbb{T}$.
(E.g., take a lattice $\Lambda$ in $\PSL(2, \mathbb{R})$ isomorphic to $\mathbb{F}_2$ (see Example E.10 in \cite{BO}).
Let $P$ denote the subgroup of upper triangular matrices in $\PSL(2, \mathbb{R})$.
Then, thanks to Theorem 9.5.3 of \cite{Zim}, Theorem 5.4.1 of \cite{BO},
and Remark 13 of \cite{Oza16}, the left translation action
$\alpha \colon \Lambda \curvearrowright \PSL(2, \mathbb{R})/P \cong \mathbb{T}$
possesses the desired properties.)
By the Pimsner--Voiculescu exact sequence (\cite{PV}, Theorem 3.5),
we obtain the exact sequence
\[ 0 \longrightarrow \mathbb{Z} \longrightarrow K_0(C(\mathbb{T})\rc \Lambda) \longrightarrow \mathbb{Z}^2.\]
From this exact sequence, one can find
a homomorphism
$K_0(C(\mathbb{T})\rc \Lambda) \rightarrow \mathbb{Z}$
sending $[1]_0$ to $1$.
It also follows from \cite{PV} that $C(\mathbb{T})\rc \Lambda$ satisfies the universal coefficient theorem (cf.~Corollary 7.2 of \cite{RS}).
By \cite{AS} and \cite{Ana87} (Theorem 4.4.3 of \cite{BO}),
$C(\mathbb{T}) \rc \Lambda$ is simple and nuclear.
Now we embed $\Gamma$ into $\Lambda$.
Then Theorem 4.1.1 of \cite{Phi}
gives a unital embedding
\[\iota \colon C(\mathbb{T}) \rc \Gamma \subset C(\mathbb{T}) \rc \Lambda \rightarrow \mathcal{O}_\infty.\]
Now define
\[\beta:= \bigotimes_{\mathbb{N}} \ad(\iota|_\Gamma) \colon \Gamma \curvearrowright \bigotimes_{\mathbb{N}}\mathcal{O}_\infty.\]
Note that $\beta$ is centrally free by Theorem 1 of \cite{Nak}.

Next choose a unital Kirchberg algebra $C$ satisfying the universal coefficient theorem with
\[K_0(C) \cong \mathbb{Z}[\Gamma]~ ({\rm as~ an~ additive~ group}),\ [1_C]_0=0,\ K_1(C)=0.\]
Take an action $\gamma \colon \Gamma \curvearrowright C$
whose induced action on $K_0(C)$
is conjugate to the left translation action on the group ring $\mathbb{Z}[\Gamma]$.
(Such an action exists by Theorem 4.1.1 of \cite{Phi}.)
Put 
\[D:= \mathcal{O}_\infty \otimes C \otimes \left(\bigotimes_{\mathbb{N}}\mathcal{O}_\infty\right).\]
We equip $D$ with the $\Gamma$-action
\[\eta := \id_{\mathcal{O}_\infty} \otimes \gamma \otimes \beta.\]
We claim that
\[K_0(D \rc \Gamma) \cong \mathbb{Z},\ K_1(D \rc \Gamma) =0.\]
We fix a free basis $S$ of $\Gamma$.
Then, thanks to the Pimsner--Voiculescu exact sequence \cite{PV},
the claim follows from the following assertion.
The additive map \[\nu \colon \bigoplus_S \mathbb{Z}[\Gamma] \rightarrow \mathbb{Z}[\Gamma]\]
given by
\[(x_s)_{s\in S} \mapsto \sum_{s\in S}(x_s - sx_s)\] satisfies
\[\ker(\nu)=0,\ {\rm coker}(\nu)\cong \mathbb{Z}.\]
(We remark that Theorem 3.5 of \cite{PV} is concentrated on the finite rank free groups.
However the infinite rank case follows from the finite rank case
as the exact sequences are compatible with the canonical inclusions $\mathbb{F}_n \subset \mathbb{F}_{n+1}$; $n\in \mathbb{N}$.)
We prove this assertion.
For $w \in \Gamma \setminus \{e\}$,
denote by $i(w) \in S\sqcup S^{-1}$ the initial alphabet of $w$ with respect to $S$.
Also, for $w\in \Gamma$,
denote by $| w|$ the length of the reduced word of $w$ (with respect to $S$).
For $a\in \mathbb{Z}[\Gamma]$,
define
\[\supp(a):=\{ s\in \Gamma: a(s)\neq 0\}.\]
To prove the first equality, assume we have $x=(x_s)_{s\in S} \in \ker(\nu) \setminus\{0\}$.
Choose $s\in S$ and $t\in \supp(x_s)$
satisfying $|t|\geq |u|$ for all $u\in \bigcup_{w \in S} \supp(x_w)$.
As $\sum_{u\in S} (x_u -u x_u)=0$,
the maximal property of $t$ forces $i(t)= s^{-1}$.
(Otherwise $|st |=|t|+1$ and $\nu(x)(st)=-x_s(t)\neq 0.$)
We then must have an element $u\in S \setminus \{s \} $ with
$t \in \supp({x_u}) \cup \supp({ux_u})$.
Again by the maximal property of $t$,
the relation $t \in \supp({x_u})$ is impossible.
Hence $t \in \supp({ux_u})$, or equivalently, $u^{-1}t \in \supp(x_u)$.
This contradicts to the choice of $t$ as $|u^{-1}t|=|t|+1$.
Thus $\ker(\nu)=0$.
To prove ${\rm coker}(\nu)\cong \mathbb{Z}$,
consider
the homomorphism $\tau_0 \colon \mathbb{Z}[\Gamma] \rightarrow \mathbb{Z}$
given by $a \mapsto \sum_{s\in \Gamma} a(s)$.
Obviously $\tau_0$ is surjective
and $\im(\nu) \subset \ker(\tau_0)$.
As $\delta_w - \delta_e \in \im(\nu)$ for all $w\in \Gamma$ (by induction on $|w|$),
we conclude $\im(\nu) = \ker(\tau_0)$.
Thus $\tau_0$ induces the desired isomorphism
${\rm coker}(\nu) \cong \mathbb{Z}$.

Let $Q=[0, 1]^\mathbb{N}$ denote the Hilbert cube.
Note that $Q$ is a contractible compact metrizable space.
Since $\Homeo(Q)$ is a Polish (hence separable) group in the topology of uniform convergence
(see \cite{Kec}, Chapter 1.9.B (3)),
one can choose a homomorphism
\[h \colon \Gamma \rightarrow \Homeo(Q)\]
with dense image.
As the diagonal action
\[\Homeo(Q) \curvearrowright \{(x_1, x_2)\in Q \times Q: x_1 \neq x_2 \}\]
is transitive (see e.g., \cite{Kle}, Theorem 4.5),
$Q$ admits no non-trivial $\Gamma$-invariant closed equivalence relation.
Thus $h$ is minimal and prime.
We equip $D \otimes C(Q)$ with the $\Gamma$-action $\eta \otimes h$.
Since $Q$ has no $\Homeo(Q)$-invariant probability measure,
the inclusion
\[D_1:= D \rc \Gamma \subset [D\otimes C(Q)] \rc \Gamma :=D_2\]
does not admit a conditional expectation.
Indeed, suppose we have a conditional 
expectation $\Psi \colon D_2 \rightarrow D_1$.
As the center of $D$ only consists of scalars,
the restriction $E \circ \Psi|_{C(Q)}$
defines a $\Gamma$-invariant probability measure on $Q$, a contradiction.
Observe that
$D$, hence also $D \otimes C(Q)$,
can be presented as the norm closure of the increasing union of nuclear $\Gamma$-\Cs-subalgebras whose underlying $\Gamma$-action is amenable
in the sense of Definition 4.3.1 in \cite{BO}.
Indeed the following sequence has the desired properties.
\[\mathcal{O}_\infty \otimes C \otimes \left(\bigotimes_{k=1}^n \mathcal{O}_\infty\right) \otimes
\iota(C(\mathbb{T}))\otimes \left(\bigotimes_{k=n+2}^\infty \mathbb{C} 1_{\mathcal{O}_\infty}\right) , n\in \mathbb{N}.\]
(Cf.~Proposition B in \cite{Suz18}.) 
Therefore by \cite{Ana87} (Theorem 4.3.4 of \cite{BO}), both $D_i$ are nuclear.
By Theorem 7.1 of \cite{BKKO}, both $D_i$ are simple.
Since both $D_i$ have $\mathcal{O}_\infty$ as a tensor factor,
they are purely infinite.
Thanks to the Pimsner--Voiculescu exact sequence \cite{PV},
both $D_i$ are KK-equivalent to $\mathcal{O}_\infty$
and the inclusion map $D_1 \rightarrow D_2$ induces a KK-equivalence.
Now choose a projection $q\in D$ generating $K_0(D_1)$ \cite{Cun}.
By the Kirchberg--Phillips classification theorem \cite{Kir94}, \cite{Phi},
both $q D_i q$ are isomorphic to $\mathcal{O}_\infty$.

We equip $A \otimes D$, $A \otimes D \otimes C(Q)$ with the $\Gamma$-action
$\id_A \otimes \eta$, $\id_A \otimes \eta \otimes h$ respectively.
Consider the inclusion
\[D_{1, A}:= (A \otimes D) \rc \Gamma \subset [A \otimes D \otimes C(Q)] \rc \Gamma:=D_{2, A}.\]
Since $h$ is prime,
Theorem 3.3 of \cite{Zac} implies the following statement.
Any intermediate $\Gamma$-\Cs-algebra
of $ A \otimes D \subset A \otimes D \otimes C(Q) $
is either equal to
$A \otimes D \otimes C(Q)$ or contained in
\[\mathcal{F}(A\otimes D, C(Q), \mathbb{C}):= \left\{ a\in A \otimes D \otimes C(Q): 
(\varphi \otimes \id_{C(Q)})(a) \in \mathbb{C}{\rm~for~all~}\varphi \in (A\otimes D)^\ast\right\}.\]
(We remark that Theorem 3.3 of \cite{Zac} is only stated for unital \Cs-algebras.
However we still can apply it as $A \otimes D$ has approximate units of projections \cite{Zha}.)
Since the inclusion $\mathbb{C} \subset C(Q)$ admits a conditional expectation, we have
$\mathcal{F}(A\otimes D, C(Q), \mathbb{C})= A \otimes D.$
Hence, by Theorem \ref{Thm:NCMainC}, $D_{1, A}\subset D_{2, A}$ has no proper intermediate \Cs-algebras.
Observe that the inclusion $D_{1, A} \subset D_{2, A}$ has no conditional expectations.
Indeed, as $A\subset D_{1, A}$ is a simple tensor factor of $D_{2, A}$,
a multiplicative domain argument shows
that any conditional expectation $\Phi \colon D_{1, A} \rightarrow D_{2, A}$
splits into $\id_A \otimes \bar{\Phi}$ for some conditional expectation
$\bar\Phi\colon D_2 \rightarrow D_1$.
We have already seen that such a $\bar{\Phi}$ does not exist.

Now consider the projection \[\tilde{q}:=1_{\mathcal{M}(A)} \otimes q\in \mathcal{M}(A\otimes D) \left(\subset \mathcal{M}(D_{1, A}) \subset \mathcal{M}(D_{2, A})\right).\]
Note that \[\tilde{q}D_{i, A}\tilde{q} \cong A \otimes qD_i q \cong A \otimes \mathcal{O}_\infty \cong A {\rm~for~}i=1, 2.\]

We next show that the inclusion
$\tilde{q}D_{1, A}\tilde{q} \subset \tilde{q}D_{2, A}\tilde{q}$
has no proper intermediate \Cs-algebras.
This follows from the next lemma.
The proof is straightforward and we leave it to the reader.
\begin{Lem}\label{Lem:corner}
Let $ \mathfrak{A} \subset \mathfrak{B}$ be a non-degenerate inclusion of \Cs-algebras.
Let $r\in \mathcal{M}(\mathfrak{A})$ be a projection.
Then the map
\[\mathfrak{C} \mapsto \Cso(\mathfrak{A}, \mathfrak{C})\]
gives an embedding
of the lattice of intermediate \Cs-algebras $\mathfrak{C}$ of $r\mathfrak{A}r \subset r\mathfrak{B}r$
into that of $\mathfrak{A} \subset \mathfrak{B}$.
Moreover, the map
$\mathfrak{D} \mapsto r\mathfrak{D}r$
gives a left inverse.
\end{Lem}
We remark that the above maps are bijective when $r$ is full in $A$
(i.e., when $\overline{\rm span}(ArA)=A$.)

\ \\
\noindent{\it Proof of Theorem \ref{Thm:End} $($continuation$)$.}
By Lemma \ref{Lem:corner}, the inclusion
$\tilde{q}D_{1, A}\tilde{q} \subset \tilde{q}D_{2, A}\tilde{q}$
has no proper intermediate \Cs-algebras.
As $D_{1, A} \subset D_{2, A}$ does not admit
a conditional expectation, neither does $\tilde{q}D_{1, A}\tilde{q} \subset \tilde{q}D_{2, A}\tilde{q}$.
Indeed, since $D$ is purely infinite,
there is an isometry $v$ in $D$
satisfying $v v^\ast \leq q$.
Put $\tilde{v}:= 1_{\mathcal{M}(A)} \otimes v$.
Then $\tilde{v}^\ast \tilde{q}D_{i, A}\tilde{q}\tilde{v} =D_{i, A}$ for $i=1, 2$.
Hence any conditional expectation $\Psi \colon \tilde{q}D_{2, A}\tilde{q} \rightarrow \tilde{q}D_{1, A}\tilde{q}$
induces a conditional expectation $\tilde{\Psi} \colon D_{2, A} \rightarrow D_{1, A}$
by the formula $\tilde{\Psi}(a):=\tilde{v}^\ast \Psi(\tilde{v} a \tilde{v}^\ast) \tilde{v}$.
This is a contradiction.

Now for $i=1, 2$, choose an isomorphism
$\sigma _i \colon \tilde{q}D_{i, A}\tilde{q} \rightarrow A.$
Denote by $\iota \colon \tilde{q}D_{1, A}\tilde{q} \rightarrow \tilde{q}D_{2, A}\tilde{q}$
the inclusion map.
The endomorphism
$\sigma := \sigma_2 \circ \iota \circ \sigma_1^{-1}$
then possesses the desired properties.
\end{proof}
\subsection*{Acknowledgements}
The author is grateful to Cyril Houdayer
for a stimulating conversation on Margulis' factor theorem \cite{Mar78} during his visiting at Universit\'{e} Paris-Sud in 2015.
He is also grateful to Yoshimichi Ueda
for helpful comments and discussions on the formulation of Theorem \ref{Thm:NCMainW}.
He is also grateful to Toshihiko Masuda for letting him know the reference \cite{TT}.
This work was supported by JSPS KAKENHI Grant-in-Aid for Young Scientists
(Start-up, No.~17H06737) and tenure track funds of Nagoya University.

\end{document}